\documentclass[aap,preprint]{imsart}

\usepackage{amsthm,amsmath,amsfonts,amssymb}
\usepackage[numbers,sort&compress]{natbib}
\usepackage{comment}
\usepackage{standalone}
\usepackage{enumitem}
\usepackage{graphicx}
\usepackage{placeins} %
\usepackage{booktabs,siunitx} %
\usepackage{textcomp}
\usepackage{lmodern} %
\usepackage[final]{microtype} %
\usepackage{bm}
\usepackage[dvipsnames]{xcolor}
\usepackage{xspace}
\usepackage{xkeyval}
\usepackage[normalem]{ulem}
\usepackage{algorithm,algorithmic}
\usepackage[colorlinks,citecolor=blue,urlcolor=blue]{hyperref}
\usepackage[capitalize]{cleveref}

\startlocaldefs
\theoremstyle{plain}
\newtheorem{theorem}{Theorem} %
\newtheorem{lemma}[theorem]{Lemma}
\newtheorem{proposition}[theorem]{Proposition}
\newtheorem{corollary}[theorem]{Corollary}

\theoremstyle{remark}
\newtheorem{assumption}{Assumption}
\newtheorem{definition}[theorem]{Definition}
\newtheorem{construction}[theorem]{Construction}
\newtheorem*{example}{Example}
\newtheorem*{remark}{Remark}

\let\textcite\cite %
\let\parencite\citep
\let\barecite\cite  %

\let\mathscr\mathcal

\DeclareMathOperator{\E}{\mathbb{E}}
\DeclareMathOperator{\Var}{\mathrm{Var}}
\DeclareMathOperator{\Cov}{\mathrm{Cov}}
\let\Pr\relax
\DeclareMathOperator{\Pr}{\mathbb{P}}

\newcommand*{\as}{a.s.\@\xspace}
\newcommand*{\iid}{i.i.d.\@\xspace}
\newcommand*{\wrt}{w.r.t.\@\xspace}

\newcommand{\xmiddle}[1]{\,\middle#1\,}
\newcommand{\xmid}{\mathop{|}} %

\makeatletter
\DeclareRobustCommand{\crefnosort}[1]{%
  \begingroup\@cref@sortfalse\cref{#1}\endgroup
}
\makeatother

\crefname{assumption}{Assumption}{Assumptions}
\crefname{construction}{Construction}{Constructions}

\makeatletter
\define@key{todo}{color}[brown]{\def\TodoColor{#1}}
\makeatother

\endlocaldefs

\begin{document}

\begin{frontmatter}
\title{Gradient Estimation\\ via Differentiable Metropolis--Hastings}
\runtitle{Differentiable Metropolis--Hastings}

\begin{aug}
\author[A]{\fnms{Gaurav}~\snm{Arya}\orcid{0000-0003-1662-3037}},
\author[B]{\fnms{Moritz}~\snm{Schauer}%
\orcid{0000-0003-3310-7915}}
\and
\author[B]{\fnms{Ruben}~\snm{Seyer}\ead[label=e3]{rubense@chalmers.se}\orcid{0000-0001-8661-3729}}
\address[A]{Department of Electrical Engineering and Computer Science, MIT}
\address[B]{Department of Mathematical Sciences,
Chalmers University of Technology and University of Gothenburg\printead[presep={,\ }]{e3}}
\end{aug}

\begin{abstract}

Metropolis--Hastings estimates intractable expectations --- can differentiating the algorithm estimate their gradients? %
The challenge is that Metropolis--Hastings trajectories are not conventionally differentiable due to the discrete accept/reject steps.
Using a technique based on recoupling chains, our method
differentiates through the Metropolis--Hastings sampler itself, allowing us to estimate gradients with respect to a parameter of otherwise intractable expectations.
Our main contribution is a proof of strong consistency and a central limit theorem for our estimator under assumptions that hold in common Bayesian inference problems.
The proofs augment the sampler chain with latent information, and formulate the estimator as a %
stopping functional of the tail of this augmented chain.
We demonstrate our method on examples of Bayesian sensitivity analysis and optimizing a random walk Metropolis proposal.
\end{abstract}

\end{frontmatter}

\section{Introduction}
Markov chain Monte Carlo (MCMC) methods are used to estimate expectations by appealing to an ergodic theorem,
\begin{equation}\label{eq:markovlln}
    \pi_\theta f = \E_{X \sim \pi_\theta}[f(X)] = \lim_{N \to \infty} \frac{1}{N} \sum_{i=0}^{N-1} f(X_i),
\end{equation}
where \(\{X_n\}_{n \in\mathbb N_0}\) is a Markov chain engineered to converge to the target distribution \(\pi_\theta\), in particular when \(\pi_\theta\) is only known up to an intractable normalizing constant.
For example, in Bayesian inference, quantities such as posterior expectation, variance, or probabilities are all integrals with respect to the posterior distribution commonly estimated in this manner. A wide class of such MCMC methods employs Metropolis--Hastings (MH) accept/reject steps to preserve the target distribution \parencite{metropolis_equation_1953,hastings_monte_1970,tierney_note_1998}.
Examples used in practice are random walk Metropolis, Metropolis-adjusted Langevin algorithm (MALA) \parencite{rossky_brownian_1978,besag_comments_1994,roberts_exponential_1996} and certain versions of Hamiltonian Monte Carlo (HMC) \parencite{duane_hybrid_1987,neal_mcmc_2011}.

If the target distribution \(\pi_\theta\) depends on a parameter \(\theta\) as above, it is natural to ask how the expectations change with respect to the parameter, quantified as the derivative at a parameter value of interest
\begin{equation*}
    \left.\frac{\partial}{\partial \theta}\E_{X \sim \pi_\theta}[f(X)]\right|_{\theta = \theta_0},
\end{equation*}
which itself must be estimated.
This problem of Monte Carlo gradient estimation arises in for example optimization of stochastic systems, sensitivity analysis and experimental design, as well as reinforcement learning and variational inference in machine learning \parencite{mohamed_monte_2020}.
Particularly within the machine learning literature, there has been interest in differentiating through Monte Carlo samplers.
Unfortunately, the MH sampler trajectory is not conventionally differentiable due to the discrete accept/reject steps.
Instead, practitioners use unadjusted methods \parencite{doucet_differentiable_2023,zhang_differentiable_2021}%
, ignore the discrete contributions and work with biased estimates \parencite{chandra_designing_2022,campbell_gradient_2021}%
, or enumerate the state space
\parencite{chandra_storytelling_2024,chandra_acting_2023}.

In this paper following up on \textcite{arya_differentiating_2023-1}, we obtain a gradient estimator by differentiating through the MH sampler itself, successfully accounting for the discreteness within the framework of \emph{stochastic perturbation analysis} (SPA) \parencite{bremaud_maximal_1992,fu_smoothed_1994,fu_conditional_1997}.
SPA-type gradient estimators weigh together several alternative trajectories of the chain.
By introducing a recoupling coupling for the alternatives \parencite{lindvall_lectures_1992,thorisson_coupling_2000}, we exploit the stochasticity to make the estimator stable and computationally tractable. %
Our main contribution is a theoretical convergence guarantee under assumptions expected to hold for models commonly considered in Bayesian inference.
Under these assumptions, we establish that our estimator essentially inherits the ergodicity of the underlying MH chain, and show that this entails strong consistency of as well as a central limit theorem for our gradient estimates.

Three major classes of gradient estimators appear in the literature \parencite{mohamed_monte_2020}:
\emph{pathwise} estimators \cite{glasserman_gradient_1991,kingma_auto-encoding_2014}, also known as the \emph{reparameterization trick}, fully extract the parameter dependency from the target distribution, yielding a very simple gradient estimator.
However, a tractable differentiable reparameterization is not available for all targets, and the method requires that \(f\) is differentiable.
\emph{Likelihood ratio} methods \parencite{glynn_likelihood_1990,glynn_likelihood_1995,glynn_likelihood_2019,rhee_lyapunov_2023}, also known as score function methods or REINFORCE, interchange derivative and expectation to directly differentiate the target density, yielding a reweighted cost \(f(x)\nabla_\theta \log \pi_\theta(x)\).
This direct approach can even be applied to unnormalized target densities if the gradient of the log-normalizing constant is additionally estimated.
However, the variance of the method deteriorates with increasing state space dimensionality \parencite{mohamed_monte_2020}.
Furthermore, \(f\) is treated like a black box, which has the benefit of wide applicability but the drawback of ignoring possible variance reduction from structural properties of \(f\).
\emph{Measure-valued derivatives} (MVD) \parencite{heidergott_measure_2006,heidergott_measure-valued_2006,heidergott_measure-valued_2008,heidergott_gradient_2010} estimate the gradient as a weighted difference of a target-dependent measure decomposition.
At the cost of simulating multiple alternatives, their difference-based approach allows nondifferentiable \(f\) while accounting for structure.
However, finding a tractable measure decomposition is nontrivial.

The performance of these gradient estimators in terms of variance and computational cost depends on the target \(\pi_\theta\) and the function \(f\) under consideration, and in general none of the estimators are universally better than any other \parencite{mohamed_monte_2020}.
Our estimator does not fit cleanly within the framework of any of the classes, aiming for the generality of the likelihood ratio method while using an alternative simulation approach similar to MVD.
Working with the MH dynamics replaces the need for target-specific tricks with a proposal-specific coupling.
We source suitable couplings for MH kernels from the recent extensive literature on this topic \parencite{wang_maximal_2021,oleary_couplings_2021,papp_new_2022,oleary_metropolis-hastings_2023}, which has been stimulated in part by the closely related application of debiased MCMC \parencite{glynn_exact_2014,jacob_unbiased_2020,middleton_unbiased_2020}, where the recoupling property is used to derive a bias correction term.
Couplings have also appeared in both SPA literature \parencite{bremaud_maximal_1992,dai_perturbation_2000} and MVD literature \parencite{heidergott_measure-valued_2006}, where they have been used for more efficient simulation, but to our knowledge the explicit use of general recoupling couplings to construct a consistent estimator for general state spaces is novel.

\subsection{Related work on Markov chain gradient estimation}
The aforementioned classes of gradient estimators have also been applied to Markov chains, and in this section, we contrast our estimator with some prior work.

As an extension of pathwise estimators, SPA has previously been used for consistent gradient estimation in among others queuing theory, financial modeling, and optimization of stochastic systems \parencite{fu_conditional_1997}.
Nevertheless, the usual formalization by generalized semi-Markov processes does not apply to our setting, as their state space is taken to be at most countable.
SPA has been extended to more general Markov chains, notably by \textcite{dai_perturbation_2000} where even couplings are used for efficient simulation, but too strong assumptions on countability, boundedness of the performance \(f\), and uniform ergodicity are imposed to obtain consistency.
We avoid these restrictions by exploiting the convergence properties of the coupling
both in simulation and in the proof of consistency.

Likelihood ratio methods have been applied to the Markov chain itself  \parencite{glynn_likelihood_1995,glynn_likelihood_2019}, estimating the gradient as ``a sequence of finite horizon expectations''.
Although their assumptions are similar to ours, they differ in one critical point: the Markov kernel is assumed to be absolutely continuous with respect to some reference measure, which is not the case for the MH kernel due to the rejection dynamics.
Furthermore, they consider a weaker notion of consistency than us, so that unlike our estimator the gradient estimator for a single infinite chain does not converge almost surely to the true derivative, but rather only in mean.

Measure-valued derivative theory has been applied to Markov chains through a ``product rule'' for kernels \parencite{heidergott_measure-valued_2006,heidergott_measure-valued_2008,heidergott_gradient_2010}.
These estimators apply to non-differentiable kernels, and use a similar structure of alternative chains called ``phantoms'' in the estimator, with a similar approach to efficient simulation of ``phantoms'' through pruning or coupling as in our algorithm.
Furthermore, our SPA weight computations are related to the corresponding measure-valued derivative for a Bernoulli random variable.
Unfortunately, the kernel factorization we introduce below to make explicit use of the coupling violates smoothness assumptions of measure-valued derivatives, so we are unable to immediately apply the theory.
Nevertheless, we hope that our recoupling trick also can be transferred to infinite-horizon measure-valued derivatives.

\subsection{Notation}
Let \((\mathsf X,\mathscr B_{\mathsf X}, \mu_0)\) be a Polish (separable completely metrizable) space equipped with the Borel \(\sigma\)-algebra and a reference measure \(\mu_0\).
This setting includes the familiar Euclidean space with Lebesgue measure but also allows for more general spaces.
We will without loss of generality assume that the parameter of interest \(\theta\) is one-dimensional and lies in some compact interval \(\Theta \subset \mathbb R\) (usually a neighborhood of some \(\theta_0\)).
The extension to multi-dimensional parameter vectors can then be done by applying the one-dimensional results for directional derivatives.

We write \(x \wedge y\) for the minimum of \(x\) and \(y\).
Given a probability measure \(\mu\), we denote the countable product
Let \(\delta_x(dy): \mathsf X \times {\mathcal B}_{\mathsf X} \to [0,1]\) be the Dirac measure, that is \(\delta_x(E) = \boldsymbol{1}_{\{x\}}(E)\) for all \(E \in {\mathcal B}_{\mathsf X}\).
For a general Markov kernel \(K(dy \xmid x) = K(x, dy): \mathsf X \times \mathscr B_{\mathsf Y} \to [0,1]\), we define by \(\mu K(A) = \int_{\mathsf X} K(A \xmid x)\mu(dx)\) the resulting measure when acting on a measure \(\mu\), and define by \(Kf(x) = \int_{\mathsf Y} f(y)K(dy \xmid x)\) the resulting function when acting on a measurable (integrable) function \(f\).
Hence kernels \(K_1(dy \xmid x): \mathsf X \times \mathscr B_{\mathsf Y} \to \mathsf [0,1], K_2(dz \xmid y): \mathsf Y \times \mathscr B_{\mathsf Z} \to [0,1]\) compose as \(K_1K_2(B \xmid x) = \int_{\mathsf Y} K_2(B \xmid y)K_1(dy \xmid x)\).
These notations may be combined and are associative.

\section{Differentiable Metropolis--Hastings estimator}
Recall that the MH chain \(\{X_n\}_{n \in \mathbb N_0}\) is defined as follows \parencite{metropolis_equation_1953,hastings_monte_1970,tierney_note_1998}:
suppose \(\pi_\theta\) is absolutely continuous with respect to \(\mu_0\), and \(g_\theta\) is an unnormalized \(\mu_0\)-density of \(\pi_\theta\).
Let \(q(dx^\circ \xmid x)\) be the proposal kernel used in the sampler which we assume has a density \(q(x^\circ \xmid x)\) with respect to \(\mu_0\) independent of \(\theta\),
and let \(\alpha_\theta(x' \xmid x)\) be the acceptance probability
\begin{equation*}
    \alpha_\theta(x' \xmid x) = \frac{g_\theta(x')q(x \xmid x')}{g_\theta(x)q(x' \xmid x)} \wedge 1.
\end{equation*}
The sampler proceeds by iterating the following kernel:
\begin{equation*}
    K_{\mathrm{MH}}(dx' \xmid x) = \alpha_\theta(x' \xmid x)q(dx' \xmid x) + \left[\int (1 - \alpha_\theta(x^\circ \xmid x))q(dx^\circ \xmid x)\right]\delta_{x}(dx').
\end{equation*}
Convergence of the sampler relies on some regularity conditions, the definitions of which are recalled in \cref{app:proofs.Kaug}.
A sufficient condition for \(\pi_\theta\)-irreducibility of the sampler is the positivity of the proposal density \(q\) for any pair of states, and a sufficient condition for aperiodicity is the ability to reject a move from any state \parencite[Section~7.3.2]{robert_monte_2004}.
The MH chain is Harris recurrent if it is \(\pi_\theta\)-irreducible \parencite[Lemma~7.3]{robert_monte_2004}.
Then, for integrable \(f\) we indeed have the Markov chain law of large numbers \eqref{eq:markovlln}
\as for any
starting state \(X_0\) \parencite[Theorem~7.4]{robert_monte_2004}.

For the estimator, we introduce a particular way to construct alternative chains to capture the impact of a perturbation of \(\theta\).
This requires that we specify how to run two chains in parallel.
Let \(q(dx^\circ \times dy^\circ \mid x, y)\) be a coupled proposal kernel, which to be a valid coupling must preserve the marginals and satisfy \(q(dx^\circ \times \mathsf X \mid x, y) = q(dx^\circ \mid x)\) and \(q(\mathsf X \times dy^\circ \mid x, y) = q(dy^\circ \mid y)\) for all \(x, y \in \mathsf X^2\).
We will use a reparameterized expression of the coupling, so that we may express \(y^\circ\) given \({x,x^\circ,y}\) as \(y^\circ = \Tilde{q}(\epsilon,x,x^\circ,y)\) with a deterministic function taking \(x,x^\circ,y\) and independent ``latent randomness'' \(\epsilon \) on \(([0,1],\mathscr B_{[0,1]},\eta)\).
The existence\footnote{Explicit knowledge of \(\Tilde{q}\) is not strictly necessary for implementation, although such a formulation is not only theoretically convenient but also practically useful \parencite{lew_adev_2023}.} of such a \(\Tilde{q}\) follows by standard disintegration and conditioning results \parencite[Lemma~4.22, Theorem~8.5]{kallenberg_foundations_2021}.

We are now ready to formulate the Differentiable MH (DMH) estimator, expanding on the algorithm first introduced by \textcite[Algorithm~2]{arya_differentiating_2023-1}.
The result is an SPA-type approach that intuitively takes the form of weighted counterfactuals (in the case where the parameter enters discretely, as in the accept/reject step).
For MH, the discrete counterfactuals come from the trajectory switching between acceptance and rejection or vice versa.
Although the trajectory itself changes discontinuously under parameter perturbations, by forcing the perturbations to occur in simulation and instead differentiating the conditional probabilities we obtain an estimate of the derivative \parencite{arya_automatic_2022,fu_conditional_1997,fu_smoothed_1994}.
We focus in the next definition on introducing the construction, with the proof of correctness contained within the later main results.
\begin{construction}[Differentiable Metropolis--Hastings estimator]\label{def:dmh}
Suppose the MH algorithm is run for \(N\) transitions starting from \(X_0\). 
The contribution to the DMH derivative estimate from a perturbation at the \(n\)th transition is the SPA estimator
\begin{equation*}
    W_n \sum_{k=1}^{N-n-1} \left[f(Y_{n,k}) - f(X_{n+k})\right],
\end{equation*}
where \(W_n\) is the weight and \(Y_{n,k}\) is the alternative chain started in \(n\) after \(k\) transitions, both quantities introduced below.
The full estimate is then the average over all transitions of these contributions.

Consider the \(n\)th state of the MH chain.
Suppose the accept/reject step is implemented by thresholding an independent \(U_n \sim \mathsf{Unif}[0,1]\) with the acceptance probability \(\alpha_\theta(X^\circ_n \mid X_n)\).
Then, inspecting \(X_n\), \(X^\circ_n\), and \(U_n\) reveals whether the primal chain accepted or rejected the proposed move, whence we obtain the weight as the derivative of the acceptance probability with a sign correction:
\begin{equation}\label{eq:weight}
    W_n = \frac{\partial}{\partial\theta}\alpha_\theta(X^\circ_n \mid X_n) \left(1 - 2 \cdot \bm{1}_{U_n \le \alpha_\theta(X^\circ_n \xmid X_n)}\right).
\end{equation}
It may however be the case that \(W_n = 0\), for example, if the proposed state is always accepted and no alternative is possible.
The alternative chain \(\{Y_{n,k}\}_{k \in \mathbb N}\) started from \(n\) can be split off by making the opposite decision to the primal in the transition:
\begin{align*}
    Y_{n,1} &= \begin{cases}
            X_n, & \text{if } U_n \le \alpha_\theta(X^\circ_n \xmid X_n)\\
            X^\circ_n, & \text{if } U_n > \alpha_\theta(X^\circ_n \xmid X_n)
        \end{cases}.
\end{align*}
Once the alternative is split, we continue using the reparameterized expression for the coupling to obtain the alternative as
\begin{equation}\label{eq:shadow}
\begin{aligned}
    Y_{n,k}^\circ &= \Tilde{q}(\epsilon_{n,k},X_{n+k},X^\circ_{n+k},Y_{n,k}),\ k=1,2,\dotsc\\
    Y_{n,k+1} &= \begin{cases}
        Y_{n,k}^\circ & \text{if } U_{n+k} \le \alpha_\theta(Y_{n,k}^\circ \xmid Y_{n,k})\\
        Y_{n,k} & \text{if } U_{n+k} > \alpha_\theta(Y_{n,k}^\circ \xmid Y_{n,k})
    \end{cases},\ k=1,2,\dotsc,
\end{aligned}
\end{equation}
where we chose to couple the accept/reject step according to a monotone coupling or common random numbers \parencite{glasserman_guidelines_1992}.
This procedure can be continued indefinitely.
\end{construction}

\subsection{Implementation details}
We implement\footnote{The code is available from \url{https://github.com/gaurav-arya/differentiable_mh}.} the DMH estimator for our numerical experiments in Julia \parencite{bezanson_julia_2017} using the package \texttt{StochasticAD.jl} \parencite{arya_automatic_2022}.
Given an implementation of a proposal coupling, this allows us to automatically obtain the procedure in \cref{def:dmh} from a straightforward implementation of the usual MH algorithm.
Nevertheless, the mathematical aspects treated in this paper are not dependent on the choice of any particular implementation.
In our previous work \textcite{arya_differentiating_2023-1} the description of the algorithm additionally incorporated specific details of weight computations and pruning (importance sampling between alternatives to deterministically bound the computational effort) present in \texttt{StochasticAD.jl}; we have here chosen to avoid such implementation details for clarity, although an efficient implementation would also consider for example the choice of pruning scheme and whether to use forward or reverse mode \parencite{radul_you_2022,becker_probabilistic_2024} automatic differentiation in gradient computations.

\section{Main results}\label{sec:main}
We begin by introducing our assumptions, which are directly linked to the different parts of the estimator.

\begin{assumption}\label{ass:MH}
    The MH chain is \(\pi_\theta\)-irreducible and \(\pi_\theta \lvert f \rvert < \infty\).
    (This implies that the MH chain is Harris recurrent and that the averages converge \as to \(\pi_\theta f\).)
\end{assumption}

\begin{assumption}\label{ass:target}
    The target unnormalized density \(g_\theta(x)\) is continuously differentiable\footnote{It is possible to weaken the differentiability assumptions at the cost of providing a more careful bound on the derivative as one runs into measurability concerns, see for example \textcite{lecuyer_unified_1990}.} in \(\theta \in \Theta\).
    Furthermore, for all \(\theta_0 \in \Theta\) we have
    \begin{equation*}
        \E_{X \sim \pi_{\theta_0}}\left[\sup_{\theta \in \Theta} \left\lvert \frac{\partial}{\partial\theta}\log \pi_\theta(X)\right\rvert^2\right] < \infty.
    \end{equation*}
\end{assumption}

\begin{assumption}\label{ass:performance}
    At least one of the following holds:
    \begin{enumerate}[label=(\alph*),ref={\ref{ass:performance}\alph*}]
        \item \(f\) is bounded on \(\mathsf X\),
            \label[assumption]{ass:performance.fbdd}
        \item the MH chain is geometrically ergodic with drift function \(V: \mathsf X \to [1, \infty]\)
            such that \(\pi_\theta V < \infty\) and \(\pi_\theta q V < \infty\),
            for which \(f\) satisfies \(\sup_{x \in \mathsf X} \lvert f(x) \rvert^{2+\gamma}/V(x) < \infty\) for some \(\gamma > 0\).
            \label[assumption]{ass:performance.geom}
    \end{enumerate}
\end{assumption}

\begin{assumption}\label{ass:coupling}
The coupling satisfies the following:
\begin{enumerate}[label=(\roman*)]
    \item It is \emph{faithful},\footnote{Terminology due to \textcite{rosenthal_faithful_1997}.} that is recoupled chains stay together after they first couple: \(q(dx^\circ \times dy^\circ \mid x, x) = q(dx^\circ \mid x)\delta_{x^\circ}(dy^\circ)\).
    \item The chains recouple in square-integrable time: let \(\tau_{x,y} = \inf_{n \in \mathbb N} \{X_n = Y_n \mid X_1 = x, Y_1 = y\}\) denote the recoupling time for parallel chains starting in \(x,y\), then we have for all \(x,y \in \mathsf X^2\) that \(\E[\tau_{x,y}^2] < \infty\).
\end{enumerate}
\end{assumption}

Essentially, \cref{ass:MH} yields the baseline convergence of MH, \cref{ass:target} yields the existence of the derivatives and moment bounds on weights in the estimator, and finally \cref{ass:performance,ass:coupling} yield moment bounds on the difference between chains and the applicability of the recoupling trick.
The assumption of geometric ergodicity is convenient and holds in many practical cases \parencite{roberts_geometric_1996,oliviero-durmus_geometric_2024}; although it could be possible to relax this with some additional work, we are interested in establishing a central limit theorem, where geometric ergodicity is a common sufficient condition \parencite{jones_markov_2004}.
We note that a drift function \(V\) with \(\pi_\theta V < \infty\) always exists for a geometrically ergodic Markov chain, such that the drift condition
\begin{equation}\label{eq:drift}
    (K_\mathrm{MH} V)(x) \le \alpha V(x) + \beta
\end{equation}
holds for some \(\alpha \in [0,1)\), \(\beta \in [0,\infty)\) \parencite{gallegos-herrada_equivalences_2023}.

The following examples describe couplings for two common variants of MH that satisfy the assumptions for use in the DMH:
\begin{example}[Independent MH]
    The simplest possible proposal is one that is independent of the current state, that is \(q(\cdot \xmid x) = q(\cdot)\).
    Then, the proposal has a trivial self-coupling \(\Tilde{q}(\epsilon,x,x^\circ,y) = x^\circ\) which always succeeds (and \(\epsilon\) can be ignored).
    If there exists a constant \(C\) such that \(\pi_\theta(x) \le C q(x)\) on the support of \(\pi_\theta\), then the sampler is uniformly ergodic \parencite[Theorem~7.8]{robert_monte_2004} and the acceptance probability in stationarity is lower-bounded by \(1/C\) \parencite[Lemma~7.9]{robert_monte_2004}.
    Since recoupling always occurs on acceptance, we deduce that the meeting times satisfy \(\Pr(\tau > t) \le (1 - \frac 1C)^t\), and hence all moments of \(\tau\) are finite.
\end{example}
\begin{example}[Gaussian Random Walk MH]
    One of the prototypical versions of MH is Gaussian Random Walk MH, which we here consider with the simple proposal \(q(\cdot \xmid x) \sim \mathsf{N}(0,\sigma^2 I)\) with fixed scale \(\sigma^2\).
    Consider a maximal reflection coupling of the proposal \parencite{wang_maximal_2021,oleary_couplings_2021}.
    This coupling is faithful and reparameterizes with \(\epsilon \sim \mathsf{Unif}[0,1]\) as
    \begin{equation*}
        \Tilde{q}(\epsilon,x,x^\circ,y) = \begin{cases}
            x^\circ, & \epsilon \le \dfrac{\varphi([x^\circ-y]/\sigma)}{\varphi([x^\circ-x]/\sigma)} \\\displaystyle
            x^\circ + \left[1 - 2\frac{\langle y-x, x^\circ-x \rangle}{\lVert y - x \rVert^2}\right](y-x), & \text{otherwise}
        \end{cases}.
    \end{equation*}
    where \(\varphi\) is the pdf for \(\mathsf{N}(0,I)\).
    This idea can be extended to a non-isotropic proposal by applying a whitening transform before coupling and then transforming back the results.
    Obtaining an analytical expression for the recoupling time is in general difficult as it will depend on how the proposals interact with the target distribution.
    A tail bound on meeting times of the form \(\Pr(\tau > t) \le C \delta^t\) for \(t, C \ge 0\), \(\delta \in (0,1)\), hence sufficient for all moments of \(\tau\) to be finite, was found by \textcite[Proposition~4]{jacob_unbiased_2020} based on the behavior of an appropriate Lyapunov function, most often obtained by assumptions on the tails of \(\pi_\theta\) \cite{roberts_geometric_1996}.
\end{example}

\begin{remark}
In a certain sense, geometric ergodicity implies \cref{ass:coupling}, as one can then manufacture a proposal kernel coupling where \(\tau_{x,y}\) has geometric tails, which by its Markovian nature can trivially be made faithful \citep[Section~3.2]{jacob_unbiased_2020}.
The important theoretical restrictions introduced by the assumptions here are therefore on the MH chain \(\{X_n\}_{n \in \mathbb N_0}\), \(f\), \(\pi_\theta\), and \(q\).
Nevertheless, for computations, it is important that the coupling is computationally tractable, and shorter recoupling times reduce the variance.
\end{remark}

\subsection{Augmented chain}\label{sec:main.Kaug}
Inspecting the DMH algorithm makes it clear that given only the MH chain \(\{X_n\}_{n \in \mathbb N_0}\) we do not have enough information to compute the coupling or the DMH estimator.
Hence, we will now \emph{augment} the transition kernel to preserve this information.
Our goal in this section is then to show that this augmented chain inherits its convergence properties from the MH chain.
A simpler augmented chain with only proposal information was analyzed similarly by \textcite{rudolf_metropolishastings_2020}, although we will use different proof strategies.

\begin{construction}[Augmented MH chain]\label{def:aug}
Consider the state space \(\mathsf X \times \mathsf X \times [0,1] \times [0,1]^{\mathbb N}\) and the product \(\sigma\)-algebra \(\mathscr B_{\mathsf X} \otimes \mathscr B_{\mathsf X} \otimes \mathscr B_{[0,1]} \otimes \mathscr B_{[0,1]}^{\otimes \mathbb N}\), where \(B_{[0,1]}^{\otimes \mathbb N}\) is the corresponding Borel \(\sigma\)-algebra on countable sequences \parencite[Lemma~1.2]{kallenberg_foundations_2021}.
Define a kernel that augments the current MH state \(x\) with the hidden information about the proposed next state, a uniform random value used to couple acceptance, and a sequence of \iid latent randomness distributed according to some \(\eta\) used to couple proposals:
\begin{equation}
    Q(dx' \times d{x^\circ}' \times du' \times d\epsilon' \mid x) = \delta_{x}(dx')q(d{x^\circ}' \xmid x)\,du'\,\eta^{\otimes \mathbb N}(d\epsilon') \label{eq:kernelN}
\end{equation}
where \(\eta^{\otimes \mathbb N}\) is the infinite product measure \parencite[Corollary~8.25]{kallenberg_foundations_2021}.
Note that that the uniform and latent sequence are sampled independently from the current state.
Next, define a kernel that given this augmented state deterministically executes the MH transition:
\begin{equation}
    R(dx' \xmid x,x^\circ,u,\epsilon) = \boldsymbol{1}_{\{u \le \alpha_\theta(x^\circ \xmid x)\}}\delta_{x^\circ}(dx') + \boldsymbol{1}_{\{u > \alpha_\theta(x^\circ \xmid x)\}}\delta_{x}(dx') \label{eq:kernelH} .
\end{equation}
Finally, define the augmented kernel
    \(K_{\mathrm{aug}} = RQ\)
which iterated yields the augmented chain \(\{(X_n, X^\circ_n, U_n, \{\epsilon_{n,k}\}_{k \in \mathbb N})\}_{n \in \mathbb N_0}\).
\end{construction}

By inspection
    \(K_\mathrm{MH} = QR\)
and hence one can interpret the new kernel \(K_\mathrm{aug}\) as instead stopping ``in the middle of'' the MH transitions.
From this characterization it follows that for \(N \in \mathbb N\) we have
\begin{equation*}%
    K_\mathrm{aug}^N = (RQ)^N = R(QR)^{N-1}Q = RK_\mathrm{MH}^{N-1}Q
\end{equation*}
and that the invariant distribution of \(K_\mathrm{aug}\) is \(\nu_\theta := \pi_\theta Q\) by a direct verification: 
\begin{equation*}
    \nu_\theta K_\mathrm{aug} = (\pi_\theta Q) (R Q) = \pi_\theta (Q R) Q = \pi_\theta Q = \nu_\theta.
\end{equation*}

We next establish the necessary technical results to inherit the desired properties and establish ergodicity; the proofs, as well as reminders of the definitions of \(\phi\)-irreducibility and Harris recurrence, are given in \cref{app:proofs.Kaug}.
\begin{proposition}[Inheriting \(\phi\)-irreducibility]\label{thm:irredaug}
    If the original MH chain is \(\phi\)-irreducible, then the augmented MH chain is \(\phi Q\)-irreducible.
\end{proposition}
\begin{proposition}[Inheriting Harris recurrence]\label{thm:harrisaug}
    If the original MH chain is Harris recurrent, then the augmented MH chain is Harris recurrent.
\end{proposition}
\begin{corollary}\label{thm:ergodicaug}
    Under \cref{ass:MH}, the invariant probability distribution \(\nu_\theta\) for \(K_\mathrm{aug}\) is unique and ergodic.
\end{corollary}

Although it is not necessary for the consistency argument that follows, and we do not assume that the original MH chain is geometrically ergodic in all cases, we also establish the inheritance of geometric ergodicity under an additional technical condition for later use in a central limit theorem.
The proof, which relies on a \emph{data processing inequality} in total variation, is also given in \cref{app:proofs.Kaug}.
\begin{proposition}[Inheriting geometric ergodicity]\label{thm:geomerg}
    Under \cref{ass:MH,ass:coupling}, additionally assuming \(q(x^\circ \xmid x) > 0\) for all \(x,x^\circ \in \mathsf X\), if the original MH chain is geometrically ergodic, then the augmented MH chain is geometrically ergodic.
\end{proposition}

\subsection{Tail functional}\label{sec:main.tail}
With the augmented chain and the reparametrization version of the coupling, we can express the DMH estimator in \cref{def:dmh} as a \emph{tail functional} of the augmented chain.
Let \(\bm{X}_n^\rightarrow = \{(X_{n+k}, X^\circ_{n+k}, U_{n+k}, \{\epsilon_{n+k,\ell}\}_{\ell \in \mathbb N})\}_{k \in \mathbb N_0}\) be the \emph{tail} from the \(n\)th augmented state.
The contribution to the DMH derivative estimate from a perturbation at the \(n\)th transition with a horizon \(m\) is
\begin{equation*}
    h_m(\bm{X}_n^\rightarrow) = W_n \sum_{k=1}^{m-1} \left[f(Y_{n,k}) - f(X_{n+k})\right]
\end{equation*}
which by construction is a deterministic function of the tail.
We define \(h = \lim_{m \to \infty} h_m\) as the \as limit.
In fact, under \cref{ass:coupling}, faithfulness implies only the terms before recoupling will be non-zero, so we may let\footnote{Note that the recoupling time was deliberately defined starting from time \(1\), so that we can form excursions from a common state at time \(0\).} \(\tau_n = \tau_{X_{n+1},Y_{n,1}}\) and write
\begin{equation*}
    h_m(\bm{X}_n^\rightarrow) = W_n \sum_{k=1}^{(\tau_n \wedge m) - 1} \left[f(Y_{n,k}) - f(X_{n+k})\right]
\end{equation*}
and similarly for \(h\).
To apply our convergence theorems, we must show that the random variable \(h\) is integrable with respect to the induced measure on sample paths \(\Pr_{\nu_\theta}\) under our assumptions.
We give the moment bound proofs in \cref{app:proofs.tail}.
\begin{lemma}[Weight moments]\label{thm:weightmoments}
    Under \cref{ass:target}, \(W_n \in L^2(\mathbb P_{\nu_\theta})\) uniformly in \(\theta\).
\end{lemma}
\begin{proposition}[Integrability of terms]\label{thm:h}
    Under \crefnosort{ass:target,ass:performance,ass:coupling},
    \(h \in L^1(\mathbb P_{\nu_\theta})\), \(\sup_{m \in \mathbb N} \lvert h_m \rvert \in L^1(\Pr_{\nu_\theta})\), and \(h_m \to h\) in \(L^1(\Pr_{\nu_\theta})\).
\end{proposition}

\begin{remark}
    The preceding results generalize to the existence of higher moments of \(\frac{\partial}{\partial \theta} \log \pi_\theta(X)\) implying the existence of higher moments of \(W_n\), although this seems less well-known for many classes of models compared to the Fisher information.
    Combining this with a stronger growth condition on \(f\) similarly yields higher moments of \(h\).
\end{remark}

\subsection{Limit theorems}\label{sec:main.limit}
Once we have established sufficient conditions for integrability, we are ready to establish our main results of unbiasedness in stationarity, strong consistency, and a central limit theorem.
All proofs in this section are given in \cref{app:proofs.limit}.

First, we show that the estimator is unbiased for a finite horizon when starting in stationarity.
Similar results have been proven under various conditions and for various types of gradient estimators in for example \textcite{arya_automatic_2022,fu_conditional_1997,heidergott_gradient_2010,seyer_differentiable_2023}.
This theorem is the critical step for the validity of the estimator and relies on a conditioning argument from SPA.
\begin{theorem}[Unbiasedness in stationarity]\label{thm:unbiasedness}
    Let \(\Pr_{\nu_\theta}\) denote the induced measure on sample paths following \(K_\mathrm{aug}\) started according to \(\nu_\theta\).
    If \cref{ass:MH,ass:target,ass:performance,ass:coupling} hold, then the DMH estimator is finite-horizon unbiased:
    \begin{equation*}%
        \left.\E_{\Pr_{\nu_\theta}}\left[\frac 1N \sum_{n=0}^{N-1}h_{N-n}(\bm{X}_n^\rightarrow)\right]\right|_{\theta = \theta_0} = \left.\frac{\partial}{\partial \theta} \E_{X \sim \pi_\theta}[f(X)]\right|_{\theta = \theta_0}
    \end{equation*}
    for all \(N \in \mathbb N\) and interior \(\theta_0 \in \Theta\).
\end{theorem}

Next, using that the augmented MH chain inherits the ergodic properties of the original MH chain, we apply Maker's ergodic theorem \parencite{maker_ergodic_1940} (recalled in \cref{app:proofs.limit}) to show that the finite-horizon version of the estimator converges to the constant infinite-horizon expectation.
This allows us to lift unbiasedness to consistency, since what we converge to must then be the desired expectation.
Indeed, establishing consistency by lifting (asymptotic) unbiasedness through a regenerative structure has been a general strategy of SPA literature \parencite{hu_strong_1990,glasserman_strongly_1991}.
\begin{theorem}\label{thm:consistency}
    If \cref{ass:MH,ass:target,ass:performance,ass:coupling} hold, then the DMH estimator is strongly consistent:
    \begin{equation*}
        \lim_{N \to \infty} \frac 1N \sum_{n=0}^{N-1} h_{N-n}(\bm{X}_n^\rightarrow) = \left.\frac{\partial}{\partial \theta} \E_{X \sim \pi_\theta}[f(X)]\right|_{\theta = \theta_0}\quad \Pr_{\bm{X}_0}\text{-\as}
    \end{equation*}
    for interior \(\theta_0 \in \Theta\) and every starting state \(\bm{X}_0 \in \mathsf X \times \mathsf X \times [0,1] \times [0,1]^{\mathbb N}\).
    In particular, for every starting state \(X_0 \in \mathsf X\) one can sample \(\bm{X}_0 \sim Q(\cdot \xmid X_0)\).
\end{theorem}

For practical purposes, guarantees on the rate of convergence are of interest.
With geometric ergodicity and stronger moment assumptions, we exploit the stopping functional structure to obtain a central limit theorem along \textcite{ben_alaya_rate_1998}.
Note that this strengthening is essentially the same as a common sufficient assumption for the classical Markov chain central limit theorem (see for example \barecite[Corollary~2]{jones_markov_2004}).
\begin{theorem}\label{thm:clt}
    If \cref{ass:MH,ass:target,ass:performance,ass:coupling} hold, the augmented MH chain is geometrically ergodic (for example by \cref{thm:geomerg}), \(\sup_{m \in \mathbb N} \lvert h_m \rvert \in L^{2+\gamma}(\Pr_{\nu_\theta})\), and \(h_m \to h\) in \(L^{2+\gamma}(\Pr_{\nu_\theta})\) for some \(\gamma > 0\), then the DMH estimator obeys a central limit theorem: there exists \(\sigma \ge 0\) such that
    \begin{equation*}
        \frac{1}{\sqrt N} \sum_{n=0}^{N-1} \left(h_{N-n}(\bm{X}_n^\rightarrow) - \left.\frac{\partial}{\partial \theta} \E_{X \sim \pi_\theta}[f(X)]\right|_{\theta = \theta_0}\right) \xrightarrow[N \to \infty]{\mathcal D} \mathsf N(0,\sigma^2)
    \end{equation*}
    (where the limiting distribution could be degenerate \(\delta_0\) if \(\sigma = 0\)) 
    for interior \(\theta_0 \in \Theta\) and every starting state \(\bm{X}_0 \in \mathsf X \times \mathsf X \times [0,1] \times [0,1]^{\mathbb N}\).
\end{theorem}
The hypotheses on \(h\) in the above theorem could, if one strengthens \cref{ass:target,ass:performance,ass:coupling} to impose the existence of moments up to order \(4 + \gamma\), be recovered by adapting the proofs of \cref{thm:weightmoments,thm:h} following a previous remark.

\section{Extensions}
We conclude the theoretical treatment with comments on a few extensions to the above results which are of interest to practitioners. %

\subsection{Performance functionals}
\Cref{ass:performance} is stated in terms of a real-valued performance function for a single state \(f(X_n)\), but this can easily be extended.
The extension to vector-valued performance functions is an immediate consequence of applying the theorems componentwise.
A further benefit of the general ergodic results used is that the extension to performance functionals \(f(X_{n},X_{n+1},\dotsc,X_{n+L})\) or a fixed finite number of states \(L\) is straightforward; the proof of \cref{thm:unbiasedness} only needs to be altered so that the partition accounts for that the \(n\)th term can be affected by perturbations in the first \(n+L\) states of the chain.
This allows us to differentiate for example autocovariance, as demonstrated in \cref{sec:proposaltuning}.

\subsection{Pruning}
For simplicity, the mathematical formalization does not encompass the pruning scheme described in the original introduction of the DMH estimator in \textcite{arya_differentiating_2023-1}.
Finite-horizon unbiasedness in stationarity of the original weight-based pruning estimator follows from the remarks by \textcite[Section~B.5]{arya_automatic_2022}, but this estimator does not directly admit the tail functional form of \(h_m\) as future alternatives also depend on the past.
However, simpler schemes such as independent subsampling\footnote{See also \parencite[Section~4.4]{heidergott_gradient_2010} for discussion of such a scheme in the context of MVD.} are clearly consistent by a simple modification of the proofs for \cref{thm:unbiasedness,thm:consistency}.

\subsection{Parameter-dependent proposals}
In the preceding, we have assumed that the proposal kernel \(q\) is independent of the parameter \(\theta\).
However, this is a priori not the case in many practical applications; often one wishes to incorporate gradient information of \(g_\theta\) into the proposal, as in MALA and HMC, or proposal parameters themselves could be the target of optimization in order to tune the sampler.
Nevertheless, depending on the application it may not be necessary to consider the sensitivity of the proposal.
When the main purpose of a gradient-informed proposal is MCMC performance, one can fix the parameter in the proposal for each run, as the marginal expectation in stationarity is independent of the proposal.

Nevertheless, the actual implementation of the algorithm can be extended easily to a truly parameter-dependent proposal with \texttt{StochasticAD.jl}, but the introduction of general non-discrete sensitivity unfortunately significantly complicates the formalization in the theoretical framework above.
We sketch some of the necessary extensions in \cref{supp:parameterdependent}.
In special cases, there may be significant simplifications that allow us to incorporate the sensitivity.
For example, if proposal admits a linear reparameterization, one can transfer the ergodicity results of the chain to pathwise gradient estimates.
This is the case in the RWMH proposal tuning example in \cref{sec:proposaltuning}.

Finally, on discrete spaces or with a proposal kernel that itself factorizes into a continuous and a discrete part, such as the Barker proposal of \textcite{livingstone_barker_2022} which incorporates target gradient information, one could instead introduce new alternatives for diverging paths due to the proposal, just as we have done for the accept-reject step.

\section{Numerical examples}\label{sec:examples}

\subsection{Prior sensitivity analysis}
In \cite{kallioinen_detecting_2024} the authors introduce a distance-based metric to diagnose sensitivity of the posterior with respect to prior or likelihood by \emph{power-scaling}, which is computed using post hoc importance sampling of the MCMC chain.
In this example, we demonstrate how one can automatically obtain an equivalent to the simpler quantity-based sensitivity metric in \textcite[Eq.~8]{kallioinen_detecting_2024} from our framework, illustrating how our general methodology can form a building block in sensitivity analysis problems.
Letting \(p(\phi)\) denote the prior probability and \(p(x \xmid \phi)\) denote the likelihood,
power-scaling the prior introduces a hyperparameter \(\theta\) into the posterior
\begin{equation*}
    p(\phi \xmid x) \propto p(x \xmid \phi) \left[p(\phi)\right]^{2^\theta}
\end{equation*}
where the scaling \(2^\theta\) is suggested \parencite[Section~2.4.3]{kallioinen_detecting_2024} to obtain a natural symmetry at zero.
The simple prior sensitivity metric then corresponds to the derivative of the posterior mean \(\frac{\partial}{\partial \theta}\E[\phi]\) at \(\theta = 0\).

We consider the case study in \textcite[Section~5.1]{kallioinen_detecting_2024}, where a prior-data conflict is diagnosed in a Bayesian linear regression model.
The goal is to predict body fat percentages given 13 simple body measurements, using the \texttt{bodyfat} dataset \parencite{johnson_fitting_1996}.
The model is
\begin{gather*}
y_i \sim \mathsf{N}(\mu_i, \sigma^2), \qquad \mu_i = \beta_0 + \sum_{k=1}^{13} \beta_k x_{ik}, \\
\beta_0 \sim t_3(0,9.2), \qquad \beta_k \sim \mathsf{N}(0,1), \qquad \sigma \sim t_3^+(0,9.2)
\end{gather*}
where the standard Gaussian priors on the regression coefficients are selected with the goal of noninformativity.

We implement the model with a \texttt{Turing.jl} specification \parencite{ge_turing_2018} which conveniently takes care of target density computations and transforming \(\sigma\) to an unbounded parameter space, only needing a small amount of bridging code to introduce the power-scaling hyperparameter.
To reproduce the original case study faithfully, we also center the covariates and set the prior mean for \(\beta_0\) to the response mean.
For the DMH, we use a Gaussian random walk proposal with a reflection coupling, adjusting for covariate scales by preconditioning with the ordinary least squares covariances.
The derivative estimate is computed with 4 chains of 350\,000 MH steps where 100\,000 steps are burn-in starting from a zero initial parameter vector.

\begin{figure}[btp]
    \centering
\begin{minipage}[t]{0.4\textwidth}
    \includegraphics[width=\textwidth]{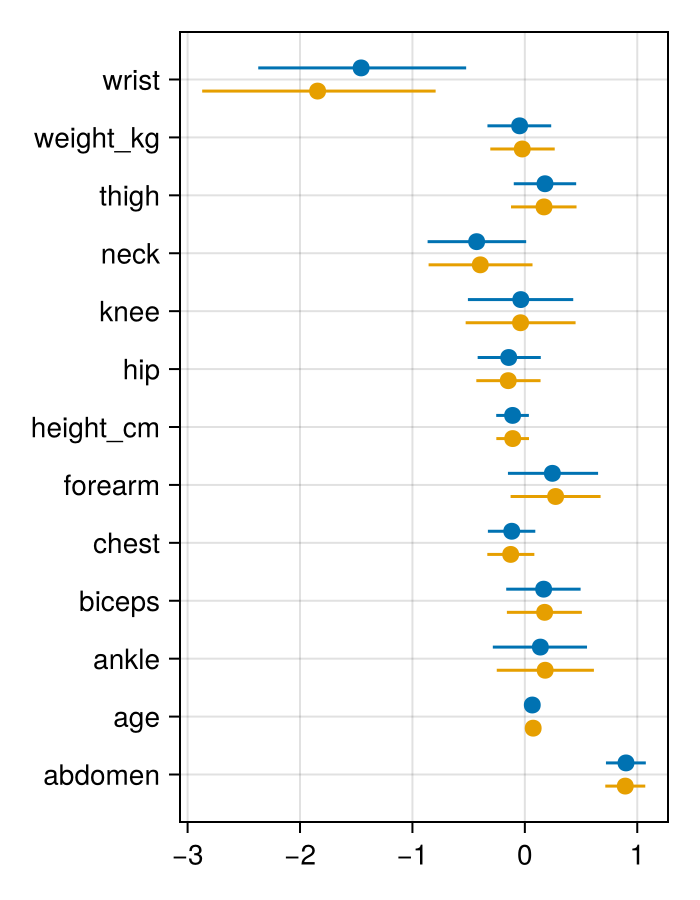}\\
    {(a) Posterior mean estimate, with 95\% credible intervals. (Intercept and \(\sigma\) omitted due to differing scales.)}
\end{minipage}
\quad
\begin{minipage}[t]{0.4\textwidth}
    \includegraphics[width=\textwidth]{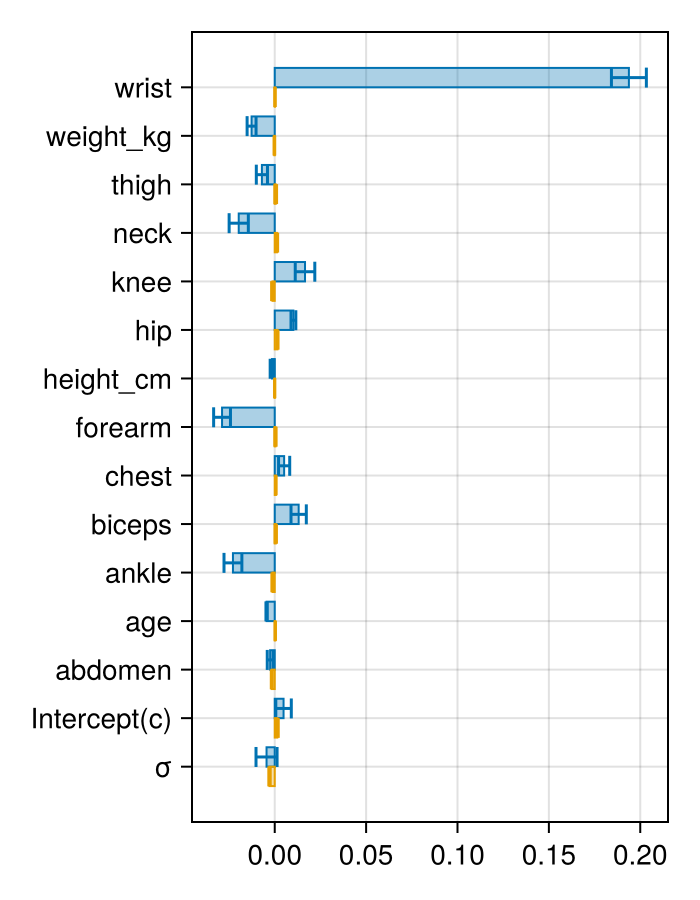}
    {(b) Sensitivity as gradient estimate for posterior means, with cross-chain standard errors.}
\end{minipage}
    \caption{Case study of prior sensitivity analysis on \texttt{bodyfat} dataset. Blue/top with original prior, orange/bottom with adjusted prior to resolve conflict.}
    \label{fig:priorsense}
\end{figure}

In \cref{fig:priorsense} we see the resulting parameter estimates and derivatives.
A clear non-zero sensitivity is reported for the \texttt{wrist} coefficient, which agrees with the results in the original case study.
To resolve the prior-data conflict, we follow \textcite{kallioinen_detecting_2024} and adjust the priors to account for the empirical scales of the covariates, so that \(\beta_k \sim \mathsf{N}(0,(2.5 s_y / s_{x_k})^2)\).
The posterior mean of the \texttt{wrist} coefficient is now subject to less shrinkage towards the prior and the sensitivity is reduced, as expected.
Although the computational properties of the DMH estimate are naturally worse than that of the original, specialized method, we emphasize that no additional theoretical derivation and little code was required to obtain these analogous results.

\subsection{Proposal kernel tuning}\label{sec:proposaltuning}
The performance of MCMC is highly dependent on the proposal distribution.
Different strategies have been used to find suitable proposals, such as adaptive MCMC methods \parencite{andrieu_tutorial_2008} or variational tuning methods \parencite{campbell_gradient_2021}.
In this example we consider minimizing the \(1\)-lag autocovariance, defined in stationarity as
\begin{equation*}
    \gamma_1 = \E_{\Pr_\nu}[(X_0 - \E_\pi[X]) (X_1 - \E_\pi[X])] = \E_{\Pr_\nu}[X_0 X_1] - \E_\pi[X]^2,
\end{equation*}
with respect to the proposal ``step size'' \(\sigma\) in Gaussian random walk MH.
For a multidimensional target, we will minimize the determinant of the corresponding cross-covariance matrix.
Autocovariance is chosen as a proxy objective for mixing speed, with the underlying intuition that an efficient sampler should yield as uncorrelated samples as possible.
This objective is related to the expected squared jump distance of the chain \parencite{sherlock_optimal_2009,sherlock_random_2010}.

We will drop the derivative of the second term in \(\gamma_1\) in our estimator, as it vanishes in stationarity, and hence our target derivative to estimate for optimization is
\(\frac{\partial}{\partial \sigma} \E_{\Pr_{\nu_\sigma}}[X_0 X_1]\).
This problem fits into our framework if we consider two of the extensions to our theoretical result since the \(1\)-lag autocovariance depends on pairs of states and the proposal is parameter-dependent.
For consistency higher moments of the target suffice; we elaborate on the theoretical details in \cref{supp:proposaltuning}.

First, we compare with the classical theory for \(d\)-dimensional isotropic Gaussian distributions, for which the optimum \(\sigma = 2.38/\sqrt{d}\) is well-known \textcite{roberts_optimal_2001}.
The proposal covariance is taken as \(\sigma^2 I\), giving us a single parameter to tune.
In \cref{fig:proposaltuning.gauss} results for a range of \(\sigma\) are shown, where each estimate is based on 100 MCMC chains with 250\,000 steps starting from zero.
The suggested objective recovers the theoretical optimum, which indicates that it is a reasonable choice for proposal tuning.
\begin{figure}[btp]
    \centering
\begin{minipage}[t]{0.32\textwidth}
    \includegraphics[width=\textwidth]{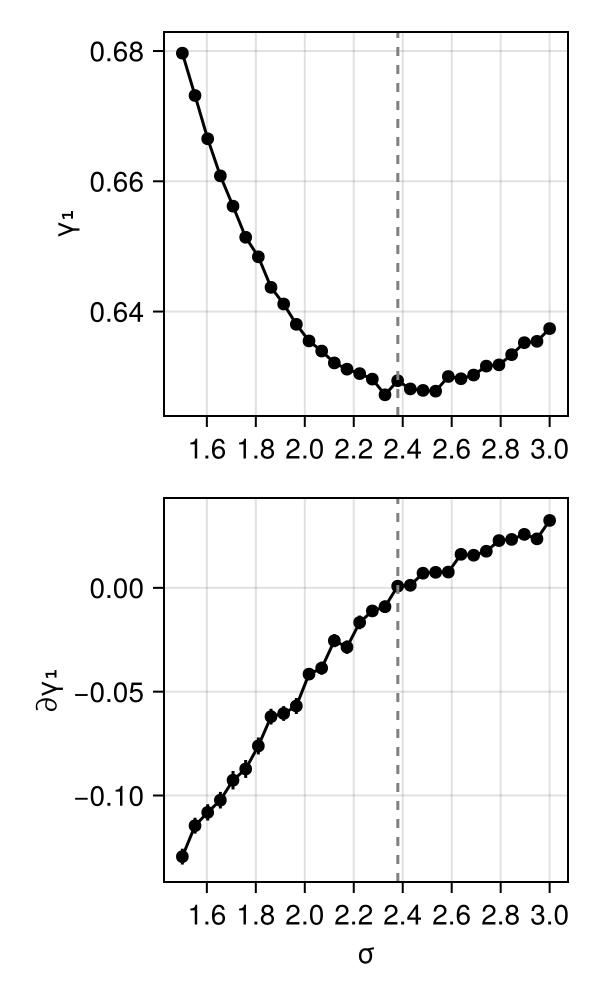}\\
    \centering{(a) \(d = 1\)}
\end{minipage}
\hfill
\begin{minipage}[t]{0.32\textwidth}
    \includegraphics[width=\textwidth]{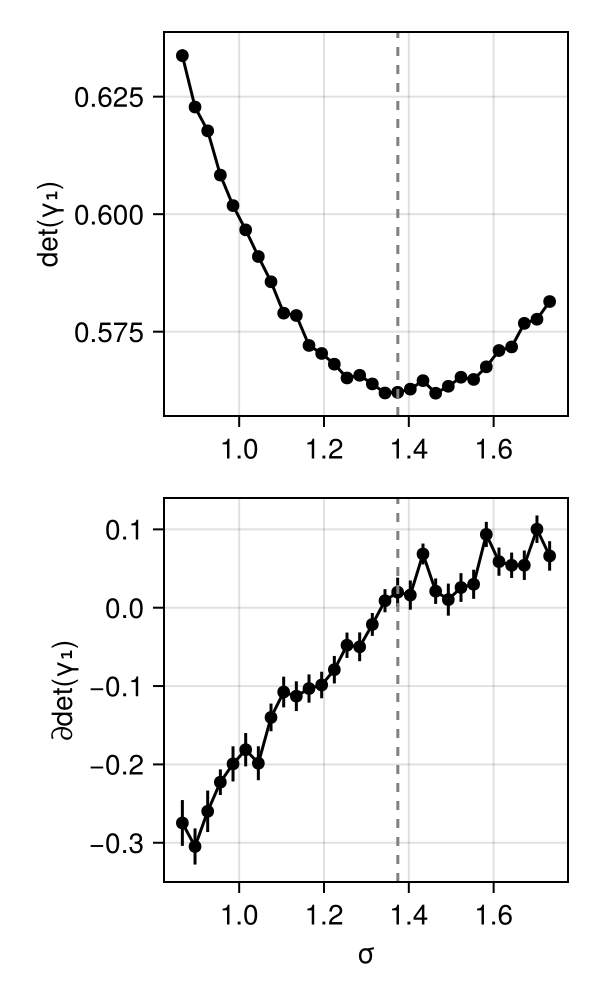}\\
    \centering{(b) \(d = 3\)}
\end{minipage}
\hfill
\begin{minipage}[t]{0.32\textwidth}
    \includegraphics[width=\textwidth]{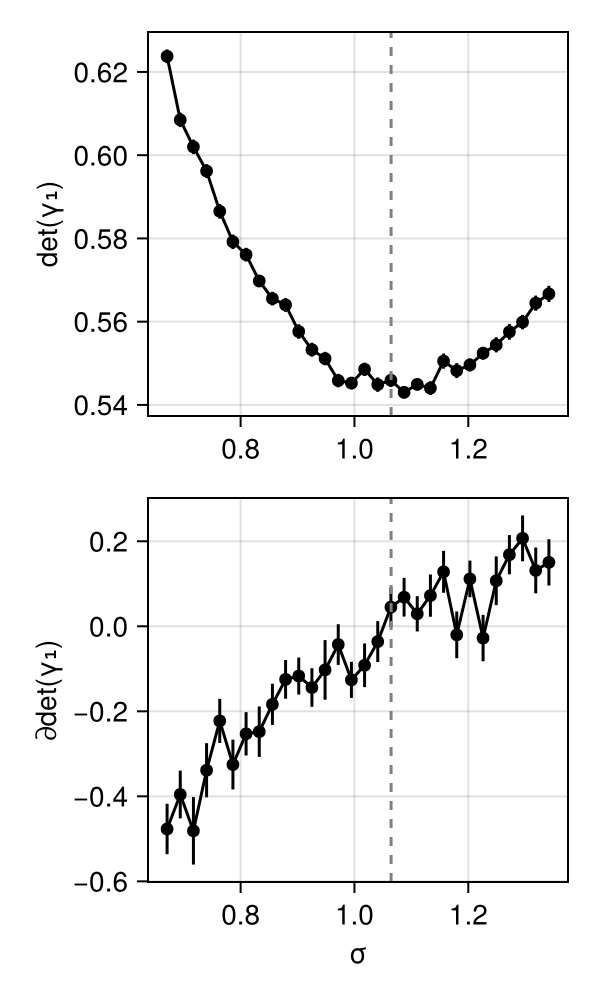}\\
    \centering{(c) \(d = 5\)}
\end{minipage} 
    \caption{Tuning an isotropic RWMH proposal targeting a \(d\)-dimensional isotropic Gaussian. Estimated 1-lag autocorrelation (top) and corresponding derivative (bottom) with cross-chain standard errors. Gray dashed line is the theoretical optimum at \(\sigma = 2.38/\sqrt d\).}
    \label{fig:proposaltuning.gauss}
\end{figure}

We now use our DMH gradient estimator in the Adam optimizer \parencite{kingma_adam_2015} to fit a full proposal covariance to some multidimensional distributions.
Each iteration will run a single chain for 250\,000 steps, and the optimizer is run for 800 iterations starting with a diagonal proposal covariance matrix.
The \(d \times d\) covariance matrix is parameterized in an unconstrained domain by the \(d(d+1)/2\) non-zero elements in its Cholesky factor, essentially learning component scales and pairwise correlations.
We show three two-dimensional distributions: a single Gaussian with correlation for which the optimum is known, a toy non-Gaussian mixture ``Dual Moon'' \parencite{campbell_gradient_2021}, and a challenging Rosenbrock-type distribution \parencite{rosenbrock_automatic_1960,goodman_ensemble_2010,pagani_ndimensional_2022}.
The learning rate of Adam is set to \(0.005\) for the first two targets and \(0.003\) for the Rosenbrock target.

The resulting proposals are shown in \cref{fig:proposaltuning.opt}.
We recover essentially the known optimal proposal in the Gaussian case and a very plausible proposal for the ``Dual Moon'' with good diagnostics that arguably performs better than the rule of thumb in terms of effective sample size.
For the Rosenbrock-type target we have slow tail exploration (compare discussion in \barecite{pagani_ndimensional_2022}) that makes the optimization slow to converge and prone to sticking in local optima, but the resulting large proposal perhaps slightly improves in the first coordinate on the rule of thumb.
Note that the acceptance rates of the resulting proposals for the latter two targets are lower than what would naïvely be targeted using acceptance rates derived from Gaussian asymptotics, as is expected \parencite{sherlock_optimal_2009}.
Full diagnostics are given in \cref{app:proposaltuning.diag}.
Finally, though this example is comparatively computationally demanding, we repeat a remark from \textcite{roberts_optimal_2001} that it is unnecessary to tune to the precise optimum, as reasonable efficiency can be achieved with sufficiently close parameters; indeed we observe diminishing returns in the optimizer indicating a quite flat objective once having found a reasonable proposal.

\begin{figure}[btp]
    \centering
\begin{minipage}[t]{0.32\textwidth}
    \includegraphics[width=\textwidth]{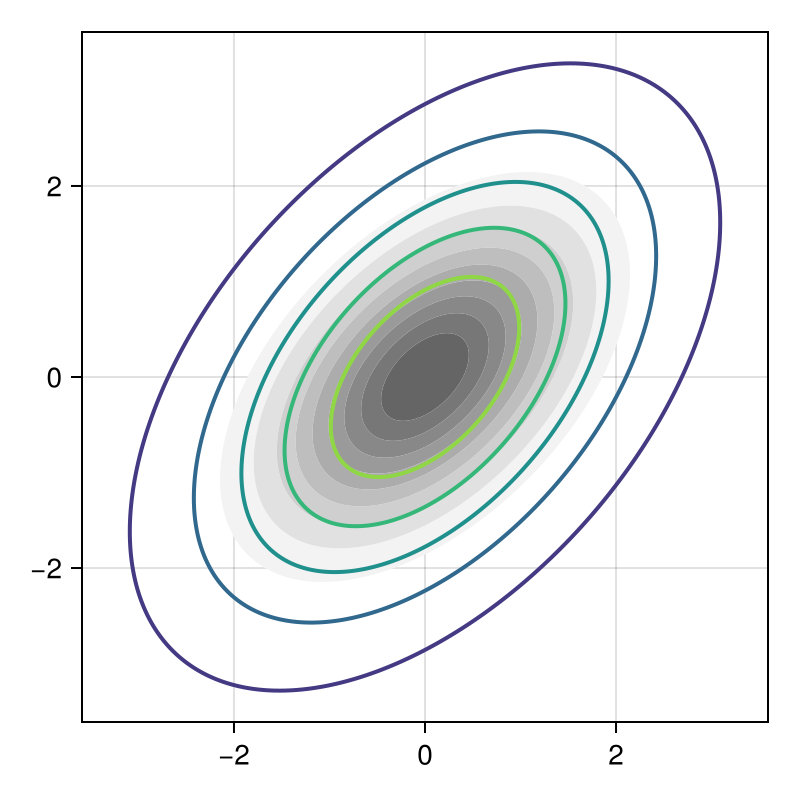}\\
    {(a) Gaussian with unit variance and correlation \(0.5\)}
\end{minipage}
\hfill
\begin{minipage}[t]{0.32\textwidth}
    \includegraphics[width=\textwidth]{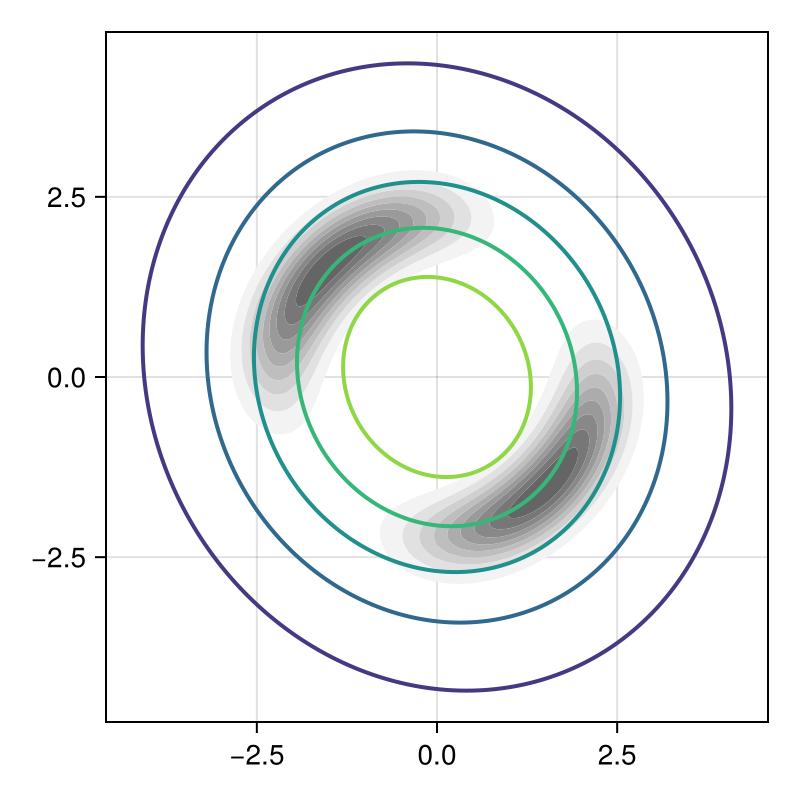}\\
    {(b) Rotated ``Dual moon'' mixture landscape}
\end{minipage}
\hfill
\begin{minipage}[t]{0.32\textwidth}
    \includegraphics[width=\textwidth]{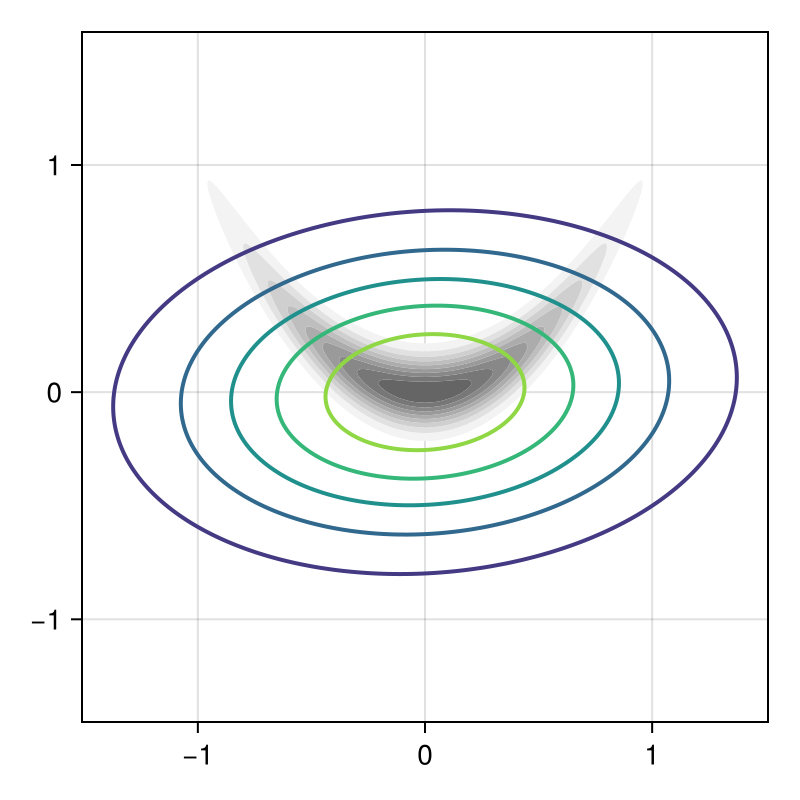}\\
    {(c) Rosenbrock \(\pi(x_1,x_2) \propto \exp(-50(x_2 - x_1^2)^2 - \frac{5}{2}x_1^2)\)}
\end{minipage} 
    \caption{Fitting a Gaussian proposal distribution for RWMH with a given target by minimizing the \(1\)-lag autocorrelation. The contour of the final proposal distribution is shown overlaid on the contour of the target distribution centered in the origin in gray.}
    \label{fig:proposaltuning.opt}
\end{figure}

\begin{appendix}
\counterwithin{theorem}{section}

\section{Proofs of Section~\ref{sec:main.Kaug}}\label{app:proofs.Kaug}
\begin{definition}[\(\phi\)-irreducibility]
 A Markov chain on \((\mathsf X, \mathscr X)\) with transition kernel \(K\) is \emph{\(\phi\)-irreducible} for a non-trivial \(\sigma\)-finite measure \(\phi\) on \((\mathsf X, \mathscr X)\) if with \(\kappa(d x' \xmid x) = \frac 12 \sum_{n = 0}^\infty 2^{-n} K^n(d x' \xmid x)\) one has for all \(x \in \mathsf X\) that \(\phi \ll \kappa(\cdot \xmid x)\), where \(\ll\) denotes absolute continuity, that is for all \(A \in \mathscr X\) we have \(\kappa(A \xmid x) = 0 \implies \phi(A) = 0\).
\end{definition}
There exist more commonly used equivalent definitions of \(\phi\)-irreducibility \parencite[Proposition~4.2.1]{meyn_markov_2009}, but this specific formulation is useful in the following result:
\begin{proof}[Proof of \cref{thm:irredaug}]
Fix $y = (x, x^\circ, u, \epsilon)$. 
From \eqref{eq:kernelH} it follows that $R(\cdot \xmid y)$ is a Dirac measure in some fixed $x' \in \mathsf X$. Let 
$\kappa = \frac 12 \sum_{n = 0}^\infty 2^{-n} K_{\text{MH}}^n$.
Then
\begin{align*}
    \int f\, d \phi Q &= \iint f(y) Q(d y \xmid \xi) \phi(d \xi)
    = \int  \tfrac{d \phi}{d \kappa(\cdot \xmid x')}(\xi) \left(\int f(y) Q(d y \xmid \xi)\right)  \kappa(d \xi \xmid x') \\
    &= \int f(y)  \tfrac{d \phi}{d \kappa(\cdot \xmid x')} d (\kappa Q)(\cdot \xmid x')
\end{align*}
as the first component of \(y\) is \(\xi\),
so $\phi Q$ has a $(\kappa Q)(\cdot \xmid x')$ density.
Thus 
\begin{equation*}
    \phi Q
    \ll (R\kappa Q)(\cdot \xmid y)
    \ll  \left(\frac 12 \sum_{n=0}^\infty 2^{-n} K_{\text{aug}}^{n + 1}\right)(\cdot \xmid y)
    \ll \left(\frac 12 \sum_{n=0}^\infty 2^{-n} K_{\text{aug}}^{n}\right)(\cdot \xmid y)
\end{equation*}
for arbitrary $y$ and the statement follows.
\end{proof}

\begin{definition}[Harris recurrence]
    A Markov chain \(\{X_n\}_{n \in \mathbb N_0}\) on \((\mathsf X, \mathscr X)\) is \emph{Harris recurrent} (\wrt \(\phi\)) if for all \(A \in \mathscr X\) such that \(\phi(A) > 0\) it holds that \(\Pr(\sum_{n=1}^\infty \bm{1}_A(X_n) = \infty) = 1\), that is the chain \as visits \(A\) infinitely many times.
\end{definition}

\begin{proposition}[\protect{\barecite[Theorem~3]{glynn_new_2023}}]
\label{prop:harris}
The chain \(\{X_n\}_{n \in \mathbb N_0}\) is Harris recurrent if and only if there exists a random time \(\tau \ge 1\) such that $X_\tau$ is independent of the starting value $x_0$ (that is, the distribution of $X_\tau$ under $\mathbb P_{x_0}$ does not depend on $x_0$). 
\end{proposition}

\begin{proof}[Proof of \cref{thm:harrisaug}]
We fix a starting state 
$y = (x, x^\circ, u, \epsilon)$ and 
write $x' = \gamma(y)$ for the map induced by the Dirac kernel $R$. Consider the Markov chain $X$ with kernel $K_\mathrm{MH}$ starting in $X_1 = x'$.
Since $X$ is Harris recurrent, and our state space is Polish and hence countably generated, we can use the Athreya--Ney--Nummelin splitting technique \parencite{nummelin_splitting_1978,athreya_limit_1978,mykland_regeneration_1995}.
Hence there exists a function $s$ with positive expectation under the invariant measure and a probability measure $\rho$ such that
\[
K_\mathrm{MH}(x, \cdot) \ge s(x)\rho(
\cdot) 
\]
which can then be used to construct a Markov chain on an augmented state space as follows: given $X_i$, sample a Bernoulli $S_i$ with success probability $s(X_i)$; conditional on success, draw $X_{i+1}$ from $\rho$, and conditional on failure, draw from the residual kernel
\[
\frac{ K_\mathrm{MH}(x, \cdot) - s(x)\rho(
\cdot)}{1-s(x)}.
\]
We may now augment the split chain setting $Y_n = (X_n, X^\circ_n, U_n, \epsilon_n)$. Here $X^\circ_n, U_n, \epsilon_n$ are draws from the \emph{conditional} law of the augmented chain as defined earlier given $X_n$ and $X_{n+1}$. Then $Y_n$ is distributed as the augmented chain in \cref{def:aug}.
The random time $\tau = \inf_i \{S_i = 1\} + 1$ renders $X_\tau$ independent from $x'$, moreover $Y_{\tau + 1}$ is independent of $y_0$. By \cref{prop:harris}, the augmented chain is Harris recurrent.
\end{proof}

\begin{proof}[Proof of \cref{thm:ergodicaug}]
    By assumption \(K_\mathrm{MH}\) is \(\pi_\theta\)-irreducible.
    From \cref{thm:irredaug} we obtain that \(K_\mathrm{aug}\) is  \(\nu_\theta\)-irreducible.
    The statement now follows by a standard Markov chain result; see for example \textcite[Theorem~5.2.6, Corollary~9.2.16]{douc_markov_2018}.
    Note that we need not worry about aperiodicity for the convergence of averages.
\end{proof}

For the final proof, we recall a useful characterization of geometric ergodicity:
\begin{proposition}[\barecite{gallegos-herrada_equivalences_2023}, Theorem~1.ii]
\label{prop:defgeomerg}
    Assume that the measurable space \((\mathsf X, \mathscr X)\) is countably generated.
    Let \(K\) be the transition kernel of a \(\phi\)-irreducible, aperiodic Markov chain on \((\mathsf X, \mathscr X)\) with invariant measure \(\pi\).
    Then the chain is geometrically ergodic if and only if there exists \(A \in \mathscr X\) with \(\pi(A) > 0\) such that for all \(x \in A\), there exist \(\rho_x < 1\) and \(C_x < \infty\) where \(\lVert K^n(\cdot \xmid x) - \pi(\cdot) \rVert_\mathrm{TV} \le C_x \rho_x^n\) for all \(n \in \mathbb N\).
\end{proposition}
\begin{proof}[Proof of \cref{thm:geomerg}]
    The technical assumptions in \cref{prop:defgeomerg} hold since the state space is assumed Polish and \(K_\mathrm{MH}\) is geometrically ergodic.
    Previously we established for \(K_\mathrm{aug}\) the invariant distribution \(\pi Q\) and the inheritance of \(\phi\)-irreducibility (\cref{thm:irredaug}).

    We begin by establishing a ``data processing inequality'' in total variation for \(Q\).
    Total variation distance between two measures \(\phi, \psi\) on \((\mathsf X, \mathscr X)\) can be characterized with the help of their possible couplings \parencite[Theorem~5.2]{lindvall_lectures_1992} as
    \begin{equation*}
        \lVert \phi(\cdot) - \psi(\cdot) \rVert_\mathrm{TV} = \min_{X \sim \phi, Y \sim \psi} \Pr(X \neq Y)
    \end{equation*}
    where there exists a coupling for which the minimum is attained.
    (Total variation is here defined to be normalized to \([0,1]\).)
    Any coupling of \(\phi, \psi\) can be then used to construct a coupling of \(\phi Q, \psi Q\) by always taking the independent components \(U, \epsilon'\) the same (a trivial self-coupling) and using our proposal coupling \(q\) in \cref{ass:coupling}.
    As the proposal coupling is faithful, it follows that 
    \begin{align*}
        \lVert \phi Q(\cdot) - \psi Q(\cdot) \rVert_\mathrm{TV} &\le \min_{\substack{X \sim \phi, Y \sim \psi \\ (X^\circ, Y^\circ) \sim q(\cdot \xmid X, Y)}} \Pr(X \neq Y \lor X^\circ \neq Y^\circ) \\
        &= \min_{\substack{X \sim \phi, Y \sim \psi \\ (X^\circ, Y^\circ) \sim q(\cdot \xmid X, Y)}} \Pr(X \neq Y) = \lVert \phi(\cdot) - \psi(\cdot) \rVert_\mathrm{TV}.
    \end{align*}

    Now, since \(K_\mathrm{MH}\) is geometrically ergodic there exists a set \(A\) with \(\pi(A) > 0\) as given by \cref{prop:defgeomerg}.
    Let \(B = A \times A \times [0,1] \times [0,1]^{\mathbb N}\).
    Then \(\pi Q(B) > 0\) by the positivity assumption on the proposal \(q\).
    Hence, for any starting state \(y \in B\) we have a corresponding \(x' = \gamma(y) \in A\) from the map induced by the Dirac kernel \(R\).
    Fix \(n \in \mathbb N\).
    Then
    \begin{align*}
        \lVert K_\mathrm{aug}^n(\cdot \xmid y) - \pi Q(\cdot) \rVert_\mathrm{TV} &= \lVert (R K_\mathrm{MH}^{n-1} Q)(\cdot \xmid y) - \pi Q(\cdot) \rVert_\mathrm{TV} \\
        &= \lVert (K_\mathrm{MH}^{n-1} Q)(\cdot \xmid x') - \pi Q(\cdot) \rVert_\mathrm{TV} \\
        &\le \lVert K_\mathrm{MH}^{n-1}(\cdot \xmid x') - \pi(\cdot) \rVert_\mathrm{TV}  \\
        &\le \max\{C_{x'},1\} \rho^{n-1} \le \frac{\max\{C_{x'},1\}}{\rho_{x'}}\rho^n_{x'}
    \end{align*}
    using the geometric ergodicity of \(K_\mathrm{MH}\) to obtain the last line of inequalities.
    Hence \(K_\mathrm{aug}\) is geometrically ergodic by \cref{prop:defgeomerg}.
\end{proof}

\section{Proofs of Section~\ref{sec:main.tail}}\label{app:proofs.tail}
\begin{proof}[Proof of \cref{thm:weightmoments}]
We will make use of the ``score expansion'' of the acceptance probability derivative
\begin{equation*}%
    \frac{\partial}{\partial\theta}\alpha_\theta(x^\circ \xmid x) = \bm{1}_{\alpha_\theta(x^\circ \xmid x) < 1} \alpha_\theta(x^\circ \xmid x) \frac{\partial}{\partial \theta} \log \frac{g_\theta(x^\circ)q(x \xmid x^\circ)}{g_\theta(x)q(x^\circ \xmid x)}.
\end{equation*}
and with \(\alpha_\theta \le 1\) and the elementary inequality 
\((a+b)^2 \le 2a^2+2b^2\)
we can expand the weights into two terms as the proposal derivatives vanish:
\begin{align*}
    \E[\lvert W_n \rvert^2] &= \E\left[\left\lvert \frac{\partial}{\partial\theta}\alpha_\theta(X^\circ_n \xmid X_n) \right\rvert^2\right] \\
    &\le 2\E\left[\left\lvert \frac{\partial}{\partial\theta}\log g_\theta(X_n)\right\rvert^2\right] + 2\E\left[\alpha_\theta(X^\circ_n \vert X_n)\left\lvert \frac{\partial}{\partial\theta}\log g_\theta(X^\circ_n)\right\rvert^2\right].%
\end{align*}
The first term is a ``translated'' version of the Fisher information of \(\pi_\theta\) owing to the missing normalization constant \(\mathcal Z_\theta\), so by the triangle inequality
\begin{equation*}
    \left\lvert \frac{\partial}{\partial\theta}\log g_\theta(X_n)\right\rvert \le \left\lvert \frac{\partial}{\partial\theta}\log \frac{g_\theta(X_n)}{\mathcal Z_\theta}\right\rvert + \left\lvert\frac{\partial}{\partial\theta}\log \mathcal Z_\theta\right\rvert
\end{equation*}
and hence is finite by hypothesis.

The second term is bounded by the using the \emph{sub-invariance} of \(\pi_\theta\) for the continuous part of \(K_\mathrm{MH}\), that is
\(\iint f(x) \alpha_\theta(x^\circ \vert x)q(dx^\circ \vert x)\pi_\theta(dx) \le \int f(x^\circ) \pi_\theta(dx^\circ)\)
for non-negative \(f\) and hence
\begin{equation*}
    \int\left(\int \left\lvert \frac{\partial}{\partial\theta}\log g_\theta(x^\circ)\right\rvert^2 \alpha_\theta(x^\circ \vert x) q(dx^\circ \vert x)\right)\pi_\theta(dx)
    \le \int \left\lvert \frac{\partial}{\partial\theta}\log g_\theta(x^\circ)\right\rvert^2 \pi_\theta(dx^\circ)
\end{equation*}
and hence is also finite.

\end{proof}
\begin{proof}[Proof of \cref{thm:h}]
There are two cases depending on the version of \cref{ass:performance}.
If we have \cref{ass:performance.fbdd}, then using the boundedness of \(f\) we have immediately
\begin{equation*}
    \lvert h(\bm{X}_0^\rightarrow) \rvert \le \lvert W_0 \rvert \sum_{k=1}^{\tau_0-1} \left\lvert f(Y_{0,k}) - f(X_{k}) \right\rvert \le 2\lvert W_0 \rvert \tau_{0} \lVert f \rVert_\infty
\end{equation*}
and hence applying Hölder's inequality
\begin{equation*}
    \E\left[\lvert h(\bm{X}_0^\rightarrow) \rvert\right] \le 2 \lVert f \rVert_\infty \E[\lvert W_0 \rvert^2]^{1/2} \E[\lvert\tau_{0}\rvert^2]^{1/2} < \infty
\end{equation*}
by the assumptions and \cref{thm:weightmoments}.

If we have \cref{ass:performance.geom} the result requires more work.
    Let \(S_T^{(X)}(f) = \sum_{k=1}^{T} \lvert f(X_{k}) \rvert\) and \(S_T^{(Y)}(f) = \sum_{k=1}^{T} \lvert f(Y_{0,k}) \rvert\).
    By the triangle inequality and Hölder's inequality, we have
    \begin{equation*}
        \E\left[\lvert h(\bm{X}_0^\rightarrow) \rvert\right] \le \E[\lvert W_0 \rvert^2]^{1/2} \left(\E\left[(S_{\tau_0}^{(Y)}(f))^2\right]^{1/2} + \E\left[(S_{\tau_0}^{(X)}(f))^2\right]^{1/2}\right)
    \end{equation*}
    and the expectation of the squared weights is bounded by \cref{thm:weightmoments},
    so it remains to bound the expectations of the squared stopped sums.
    Under our hypotheses, the marginal chains for \(X\) and \(Y\) are geometrically ergodic MH chains, and without loss of generality, they have initial distributions \(\pi_\theta\) and \(\pi_\theta q\), respectively.
    Hence, the drift condition \eqref{eq:drift} holds.
    We also have \(f(x)^2 \le C_1 V(x)\).
    By Cauchy-Schwarz, the drift condition also holds for \(\sqrt V\) as
    \begin{equation*}
        K_\mathrm{MH} \sqrt{V(x)} \le \sqrt{K_\mathrm{MH} V(x)} \le \sqrt{\alpha V(x) + \beta} \le \sqrt{\alpha} \sqrt{V(x)} + \sqrt{\beta}.
    \end{equation*}
    In this setting, we can apply a result by \textcite[Theorem~3.1]{alsmeyer_existence_2003}, in a special form that follows from its proof in the case of exponent \(2\), to deduce for some \(C_2 > 0\) that
    \begin{align*}
        \E\left[S_{\tau_0}^{(X)}(f)^2\right] &\le C_2 \left(\E\left[S_{\tau_0}^{(X)}(V)\right] + \pi_\theta(V) + \E[\tau_0^2]\right),\\
        \E\left[S_{\tau_0}^{(Y)}(f)^2\right] &\le C_2 \left(\E\left[S_{\tau_0}^{(Y)}(V)\right] + \pi_\theta q (V) + \pi_\theta q K_\mathrm{MH} (V) + \E[\tau_0^2]\right).
    \end{align*}
    where in particular using the drift condition yields
    \begin{equation*}
        \pi_\theta q K_\mathrm{MH}(V) \le \alpha \pi_\theta q(V) + \beta < \infty.
    \end{equation*}
    Next, using a result by \textcite[Lemma~2.2]{moustakides_extension_1999} we have
    \begin{align*}
        \E\left[S_{\tau_0}^{(X)}(V)\right] &\le \frac{1}{1 - \alpha} \left(\pi_\theta(V) + \beta\E[\tau_0]\right) < \infty,\\
        \E\left[S_{\tau_0}^{(Y)}(V)\right] &\le \frac{1}{1 - \alpha} \left(\pi_\theta q(V) + \beta\E[\tau_0]\right) < \infty,
    \end{align*}
    both finite by hypothesis.
    The first statement follows.

    To obtain the remaining statements, apply the triangle inequality to \(\lvert h_m \rvert\) and note that \(\tau_0 \wedge m \le \tau_0\) allows using the same bounds as before to obtain integrable domination uniformly in \(m\); then use dominated convergence.
\end{proof}

\section{Proofs of Section~\ref{sec:main.limit}}\label{app:proofs.limit}
\begin{proof}[Proof of \cref{thm:unbiasedness}]
    Fix \(N \in \mathbb N\) and interior \(\theta_0 \in \Theta\).
    To avoid excessive indices, let \(\nu := \nu_{\theta_0}\) and let \(\E\) denote \(\E_{\Pr_\nu}\) as we will consider the path measure fixed and study the impact by perturbing the kernel itself. %
    By stationarity and linearity of expectation
    \begin{equation*}
        \left.\E_{X \sim \pi_\theta}[f(X)]\right|_{\theta=\theta_0} = \E\left[\frac 1N \sum_{n=0}^{N-1}f(X_n)\right] = \frac 1N \sum_{n=0}^{N-1}\E\left[f(X_k)\right]
    \end{equation*}
    and so using the linearity of derivatives we will consider each state separately.
    
    The dependence on \(\theta\) in \(K_\mathrm{aug}\) is entirely structural in the sense that perturbations come from acceptance probabilities in \(R\).
    If we consider
    \begin{equation}\label{eq:HMVD}
        (R_{\theta + \epsilon} - R_{\theta-\epsilon})f(x,x^\circ,u,\epsilon') = \begin{cases}
            f(x^\circ) - f(x) & \alpha_{\theta-\epsilon}(x^\circ \vert x) < u \le \alpha_{\theta+\epsilon}(x^\circ \vert x) \\
            f(x) - f(x^\circ) & \alpha_{\theta+\epsilon}(x^\circ \vert x) < u \le \alpha_{\theta-\epsilon}(x^\circ \vert x)\\
            0 & \text{otherwise}
        \end{cases}
    \end{equation}
    it is clear that a specific interval of \(u\) causes the path to diverge in a manner dependent on the sign of \(\frac{\partial \alpha}{\partial \theta}\).
    This will allow us to form events on which perturbations occur.

    Hence, we use a filtered Monte Carlo approach \parencite{fu_conditional_1997} with two versions of the chain running in parallel for \(\theta_0 \pm \epsilon\) according to the coupling we have defined in \eqref{eq:shadow}.
    Let \(A_k(\epsilon)\) be the event on which no perturbation occurs through the \(k\)th state, and let \(B_k(\epsilon)\) be the event on which the first perturbation occurs \emph{leaving} the \(k\)th state.
    By definition \(A_{k-1}(\epsilon) \setminus A_k(\epsilon) = B_{k - 1}(\epsilon)\).
    Hence \(B_0(\epsilon), B_1(\epsilon), \dotsc, B_{n-1}(\epsilon)\), and \(A_n(\epsilon)\) form a partition for any fixed \(n \in \mathbb N\).

    Fixing \(n \in \{1,\dotsc,N-1\}\), writing \(X_n^{(\theta_0+\epsilon)}, X_n^{(\theta_0-\epsilon)}\) for the alternative paths described above, we partition the derivative written as a central difference:
    \begin{align}
        \frac{\partial}{\partial \theta} \E\left[f(X_n)\right] &= \lim_{\epsilon \to 0} \frac{1}{2\epsilon}\E\left[f(X_n^{(\theta_0+\epsilon)})-f(X_n^{(\theta_0-\epsilon)})\right] \notag\\
        &= \lim_{\epsilon \to 0} \E\left[\frac{f(X_n^{(\theta_0+\epsilon)})-f(X_n^{(\theta_0-\epsilon)})}{2\epsilon} \bm{1}_{A_n(\epsilon)}\right] \label{eq:unbiasedness.IPA}\\
        &\quad + \sum_{k = 0}^{n-1}\lim_{\epsilon \to 0} \E\left[\left(f(X_n^{(\theta_0+\epsilon)})-f(X_n^{(\theta_0-\epsilon)})\right)\frac{\bm{1}_{B_k(\epsilon)}}{2\epsilon}\right]. \label{eq:unbiasedness.SPA}
    \end{align}
    In the first term \eqref{eq:unbiasedness.IPA}, where no perturbation occurs on the event, the difference is zero by definition.
    It remains to compute the limit in the second term \eqref{eq:unbiasedness.SPA}.
    Consider a fixed \(k \in \{0,1,\dotsc,n-1\}\), then condition on \(B_k(\epsilon)\) and the states \(X_k, X_k^\circ\) to obtain
    \begin{align*}
        &\E\left[\left(f(X_n^{(\theta_0+\epsilon)})-f(X_n^{(\theta_0-\epsilon)})\right)\frac{\bm{1}_{B_k(\epsilon)}}{2\epsilon}\right] \\
        &= \E\Bigg[\underbrace{\E\left[f(X_n^{(\theta_0+\epsilon)})-f(X_n^{(\theta_0-\epsilon)})\middle\vert B_k(\epsilon), X_k, X_k^\circ\right]}_{=: \Delta_{n, k, \epsilon}}\underbrace{\frac{\Pr(B_k(\epsilon) \vert X_k, X_k^\circ)}{2\epsilon}}_{=: P_{k,\epsilon}}\Bigg]
    \end{align*}

    First, we note that the conditional difference is integrably dominated, either trivially by \cref{ass:performance.fbdd} or by using the \(\theta\)-uniform drift condition \eqref{eq:drift} from \cref{ass:performance.geom} to obtain
    \begin{equation*}%
        \left\lvert f(X_n^{(\theta_0+\epsilon)})-f(X_n^{(\theta_0-\epsilon)}) \right\rvert \le 2C_f\left(\alpha^{n-k}(V(X_k^\circ) + V(X_k)) + \frac{\beta}{1-\alpha}\right).
    \end{equation*}
    The \as limit of the conditional difference is established by case analysis on the sign of \(\frac{\partial \alpha}{\partial \theta}\) and \(U_k\).
    By forcing the perturbation according to \eqref{eq:HMVD}, thus forcing \(B_k(\epsilon)\) to occur so that the limit is only in the dynamics, we establish for the different cases
    \begin{equation*}
        \begin{cases}
            -\E\left[f(Y_{k,n-k}^{(\theta_0-\epsilon)})-f(X_n^{(\theta_0+\epsilon)})\xmiddle| X_k, X_k^\circ\right] & \frac{\partial \alpha}{\partial \theta} > 0 \text{ and } U_k \le \alpha_{\theta_0}(X^\circ_k \xmid X_k)\\
            +\E\left[f(Y_{k,n-k}^{(\theta_0+\epsilon)})-f(X_n^{(\theta_0-\epsilon)})\xmiddle\vert X_k, X_k^\circ\right] & \frac{\partial \alpha}{\partial \theta} > 0 \text{ and } U_k > \alpha_{\theta_0}(X^\circ_k \vert X_k)\\
            +\E\left[f(Y_{k,n-k}^{(\theta_0+\epsilon)})-f(X_n^{(\theta_0-\epsilon)})\middle\vert X_k, X_k^\circ\right] & \frac{\partial \alpha}{\partial \theta} \le 0 \text{ and } U_k \le \alpha_{\theta_0}(X^\circ_k \vert X_k)\\
            -\E\left[f(Y_{k,n-k}^{(\theta_0-\epsilon)})-f(X_n^{(\theta_0+\epsilon)})\middle\vert X_k, X_k^\circ\right] & \frac{\partial \alpha}{\partial \theta} \le 0 \text{ and } U_k > \alpha_{\theta_0}(X^\circ_k \vert X_k)
        \end{cases}
    \end{equation*}
    all of which by conditional dominated convergence have the same limit except for the sign \(S_k = \left(1 - 2 \cdot \bm{1}_{U_k \le \alpha_{\theta_0}(X^\circ_k \vert X_k)}\right)\mathrm{sign}\left(\frac{\partial}{\partial\theta} \alpha_{\theta_0}(X^\circ_k \vert X_k)\right)\), which gives the expression
    \begin{equation*}
       \Delta_{n, k, \epsilon} \xrightarrow[\epsilon \to 0]{\text{\as}} \E\left[(f(Y_{k,n-k})-f(X_n)) S_k \middle\vert X_k, X_k^\circ\right].
    \end{equation*}

    For the second factor, the characterization in \eqref{eq:HMVD} and \cref{ass:target} imply
    \begin{equation*}
        P_{k,\epsilon} = \frac{\lvert \alpha_{\theta_0+\epsilon}(X^\circ_k \vert X_k) - \alpha_{\theta_0-\epsilon}(X^\circ_k \vert X_k)\rvert}{2\epsilon}
        \xrightarrow[\epsilon \to 0]{\text{\as}} \left\lvert\frac{\partial}{\partial\theta} \alpha_{\theta_0}(X^\circ_k \vert X_k) \right\rvert.
    \end{equation*}
    By moving the sign \(S_k\) to the second factor,
    this establishes that the \as limit of the product \(\Delta_{n,k,\epsilon}P_{k,\epsilon}\) is in fact \(\E\left[W_k\left(f(Y_{k,n-k}) - f(X_n)\right)\middle\vert X_k, X_k^\circ\right]\) with our previously defined \(W_k\).
    It remains to show that we can apply dominated convergence to pass to a limit in mean.

    Hence, we now establish that the product is uniformly integrable.
    We need to show
    \begin{equation*}
        \lim_{K \to \infty} \sup_{\epsilon \ge 0} \E\left[\lvert \Delta_{n,k,\epsilon}P_{k,\epsilon} \rvert \bm{1}_{\lvert \Delta_{n,k,\epsilon}P_{k,\epsilon} \rvert > K}\right] = 0 .
    \end{equation*}
    Apply Hölder's inequality for fixed \(K, \epsilon\) to obtain
    \begin{align*}
        &\E\left[\lvert \Delta_{n,k,\epsilon}P_{k,\epsilon} \rvert \bm{1}_{\lvert \Delta_{n,k,\epsilon}P_{k,\epsilon} \rvert > K}\right]\\
        &\quad\le \E\left[\lvert \Delta_{n,k,\epsilon}\rvert^{2+\gamma}\right]^{\frac{1}{2 + \gamma}} \E\left[\lvert P_{k,\epsilon} \rvert^2\right]^{\frac 12} \Pr(\lvert \Delta_{n,k,\epsilon}P_{k,\epsilon} \rvert > K)^{\frac{\gamma}{4 + 2\gamma}}.
    \end{align*}
    The first factor is uniformly bounded by the same argument as before; either trivially by \cref{ass:performance.fbdd} or by using the \(\theta\)-uniform drift condition \eqref{eq:drift} from \cref{ass:performance.geom} to obtain the integrable bound
    \begin{equation*}
        \left\lvert f(X_n^{(\theta_0+\epsilon)})-f(X_n^{(\theta_0-\epsilon)}) \right\rvert^{2 + \gamma} \le 2^{2+\gamma}C_f\left(\alpha^{n-k}(V(X_k^\circ) + V(X_k)) + \frac{\beta}{1-\alpha}\right).
    \end{equation*}
    The second factor is uniformly bounded by hypothesis; applying the mean value theorem \parencite[Theorem~1.2]{fu_conditional_1997} and using the uniformity in \cref{thm:weightmoments} we deduce
    \begin{equation*}
        \sup_{\epsilon \ge 0} \E\left[\lvert P_{k,\epsilon} \rvert^2\right]^{\frac 12} \le \E\left[\sup_{\theta \in \Theta} \lvert W_k \rvert^2\right]^{\frac 12} < \infty.
    \end{equation*}
    Finally, the third factor has
    \begin{equation*}
        \sup_{\epsilon \ge 0} \Pr(\lvert \Delta_{n,k,\epsilon}P_{k,\epsilon} \rvert > K)^{\frac{\gamma}{4 + 2\gamma}} \le \Pr\left(\sup_{\epsilon \ge 0} \lvert \Delta_{n,k,\epsilon}P_{k,\epsilon} \rvert > K\right)^{\frac{\gamma}{4 + 2\gamma}} \xrightarrow[K \to \infty]{} 0
    \end{equation*}
    and uniform integrability follows.
    
    Hence, summing over \(k\),
    \begin{equation*}
        \frac{\partial}{\partial \theta} \E\left[f(X_n)\right] = \E\left[\sum_{k = 0}^{n-1} W_k\left(f(Y_{k,n-k}) - f(X_n)\right)\right].
    \end{equation*}
    To recover the tail functionals, we sum over all \(n\) and transpose the indices from the average over past perturbations affecting the \(n\)th state to the average over future perturbations beginning from the \(k\)th state.
    Since no perturbation can occur for the initial state, and the transition out of the final state has no impact in this window, we obtain
    \begin{align*}
        \frac{\partial}{\partial \theta}\left.\E_{X \sim \pi_\theta}[f(X)]\right|_{\theta=\theta_0} &= \frac 1N \sum_{n=0}^{N-1}\frac{\partial}{\partial \theta}\E\left[f(X_n)\right]\\
        &= \E\left[\frac 1N \sum_{n=1}^{N-1} \sum_{k= 0}^{n-1} W_k\left(f(Y_{k,n-k}) - f(X_n)\right)\right]\\
        &= \E\left[\frac 1N \sum_{k = 0}^{N-2} \sum_{\ell=1}^{N-k-1} W_k\left(f(Y_{k,\ell}) - f(X_{k + \ell})\right)\right]\\
        &= \E\left[\frac 1N \sum_{k = 0}^{N-1} h_{N-k}(\bm{X}_k^\rightarrow)\right]
    \end{align*}
    as required.
\end{proof}

The usual Markov chain law of large numbers is stated for functions of the ``current'' state.
However, the general Birkhoff ergodic theorem also applies to tail functionals such as our \(h\) under consideration.
\begin{theorem}[\protect{\barecite[Theorems~5.1.8, 5.2.9]{douc_markov_2018}}]\label{thm:birkhoff}
    Let \(K\) be a Markov kernel on \(\mathsf Z\), which admits an invariant probability measure \(\nu\) such that \((\mathsf{Z}^{\mathbb N}, \mathscr{Z}^{\otimes\mathbb N}, \Pr_{\nu},T)\) is an ergodic dynamical system, where \(\Pr_\nu\) denotes the induced measure on sample paths following \(K\) started according to \(\nu\), and \(T\) is the sequence shift.
    Let \(h \in L^1(\Pr_{\nu})\).
    Then for \(\nu\)-almost every \(z \in \mathsf Z\),
    \begin{equation*}%
        \lim_{N \to \infty} \frac{1}{N} \sum_{n=0}^{N-1} h \circ T^n = \E_{\Pr_\nu}[h]\quad \Pr_z\text{-\as}
    \end{equation*}
    with convergence also in \(L^1(\Pr_\nu)\).
\end{theorem}
The infinite-horizon result also extends to the finite-horizon case by use of Maker's ergodic theorem:
\begin{theorem}[\protect{following \barecite[Theorem~10.4.1]{maker_ergodic_1940}; \barecite{hochman_notes_2017}}]\label{thm:maker}
    In the setting of \cref{thm:birkhoff}, let \(\sup_{m \in \mathbb N} \lvert h_m \rvert \in L^1(\Pr_{\nu})\) and \(h_m \to h\) \(\Pr_{\nu}\)-\as.
    Then for \(\nu\)-almost every \(z \in \mathsf Z\),
    \begin{equation*}%
        \lim_{N \to \infty} \frac{1}{N} \sum_{n=0}^{N-1} h_{N-n} \circ T^n = \E_{\Pr_\nu}[h]\quad \Pr_z\text{-\as}
    \end{equation*}
    with convergence also in \(L^1(\Pr_\nu)\).
\end{theorem}

\begin{corollary}[\protect{using \barecite[Proposition~17.1.6]{meyn_markov_2009}}]\label{thm:harrisae}
    If in \cref{thm:birkhoff,thm:maker} the kernel \(K\) is Harris recurrent, then ``\(\nu\)-almost every'' is strengthened to ``every'' \(z \in \mathsf Z\).
\end{corollary}

\begin{proof}[Proof of \cref{thm:consistency}]
    Under the hypotheses, convergence of the estimator \(\mathbb P_{\bm{X}_0}\)-\as and in \(L^1(\Pr_{\nu_{\theta_0}})\) follows by \cref{thm:harrisaug,thm:maker,thm:harrisae}.
    The sampling initialization regime is obvious by construction of \({\nu_{\theta_0}}\).
    It remains to prove that the constant limit \(\E_{\Pr_{\nu_{\theta_0}}}[h]\) is the correct one.
    However, under the hypotheses \cref{thm:unbiasedness} holds, and taking the limit as \(n \to \infty\) of both sides in \cref{thm:unbiasedness} yields
    \begin{equation*}
        \E_{\Pr_{\nu_{\theta_0}}}\left[h\right] = \left.\frac{\partial}{\partial \theta} \E_{X \sim \pi_\theta}[f(X)]\right|_{\theta = \theta_0}
    \end{equation*}
    as required.
\end{proof}

\begin{proof}[Proof of \cref{thm:clt}]
This is an almost\footnote{Although the theorem is stated for an Euclidean Borel state space, nothing in the proofs depends explicitly on these values and we may extend to our setting.} direct application of \textcite[Theorem~5]{ben_alaya_rate_1998}.
Consider first a central limit theorem for the infinite-horizon version \(h\).
The mixing condition follows by geometric ergodicity as \textcite[Section~6.1.1]{ben_alaya_rate_1998} implies the \(\alpha\)-mixing coefficients decay geometrically fast \(\Pr_{\nu_{\theta_0}}\)-\as, the moment conditions on the stopping time follow by \cref{ass:coupling} according to \textcite[Remark~3.1.1b]{ben_alaya_rate_1998}, and the moment conditions on the functional follow by hypothesis.
The result holds for any starting distribution by Harris recurrence from \cref{thm:harrisaug} using \textcite[Proposition~17.1.6]{meyn_markov_2009}. 

It remains to pass from an infinite horizon to a finite horizon.
Consider the difference
\begin{equation*}
    \frac{1}{\sqrt N}\sum_{n=0}^{N-1} \left(h_{N-n}(\bm{X}_n^\rightarrow) - h(\bm{X}_n^\rightarrow)\right) %
\end{equation*}
and observe that it is zero-mean in stationarity due to unbiasedness.
Furthermore, using the stationarity it follows (using the shift notation \(h(\bm{X}_n^\rightarrow) = h \circ T^n\) for conciseness)
\begin{align*}
    &\Var\left(\frac{1}{\sqrt N}\sum_{n=0}^{N-1} \left(h(\bm{X}_n^\rightarrow) - h_{N-n}(\bm{X}_n^\rightarrow)\right)\right) \\
    &= \frac{1}{N}\sum_{n=0}^{N-1} \Var\left(h - h_{N-n}\right)
    + \frac{2}{N}\sum_{n=0}^{N-1}\sum_{k=n+1}^{N-1} \Cov\left(h - h_{N-n},\middle. (h - h_{N-k}) \circ T^{k-n}\right)
    \\
     &= \frac{1}{N}\sum_{m=1}^{N} \Var\left(h - h_{m}\right) + \frac{2}{N}\sum_{m=1}^{N}\sum_{\ell=1}^{m-1} \Cov\left(h - h_{m},\middle. (h - h_{m-\ell}) \circ T^{\ell}\right).
\end{align*}
The first term tends to zero as \(N \to \infty\), since \(\Var(h - h_m) \to 0\) by hypothesis and it is the Cesàro mean of this sequence.
The second term also tends to zero as \(N \to \infty\), although the result is more complicated and requires a variant of \textcite[Theorem~3]{ben_alaya_rate_1998} with the uniform moment bound hypothesis to deduce
\begin{align*}
    \lvert \Cov\left(h - h_{m},\middle. (h - h_{m-\ell}) \circ T^{\ell}\right) \rvert \le A_m \alpha^{\gamma/(2+\gamma)}(\ell - [\ell/2]) + \frac{B_m}{[\ell/2]^{2(1 + \gamma)/(2 + \gamma)}}
\end{align*}
where \(\alpha\) denotes the aforementioned mixing coefficients, \(A_m = 16 \lVert h - h_m \rVert_{2 + \gamma} (\lVert h \rVert_{2+\gamma} + \sup_{n \in \mathbb N} \lVert h_n \rVert_{2+\gamma})\), and \(B_m = \lVert h - h_m \rVert_{2 + \gamma} (\lVert h \rVert_{2+\gamma} + \sup_{n \in \mathbb N} \lVert h_n \rVert_{2+\gamma}) \E(\tau^2)^{(1+\gamma)/(2+\gamma)}\).
Hence, the covariances are summable in \(\ell\), and it follows for the sum
\begin{align*}
    &\left\lvert \sum_{\ell = 1}^{m-1} \Cov\left(h - h_{m},\middle. (h - h_{m-\ell}) \circ T^{\ell}\right) \right\rvert \\
    &\quad\le A_m \sum_{\ell = 1}^{\infty} \alpha^{\gamma/(2+\gamma)}(\ell - [\ell/2]) + B_m \sum_{\ell = 1}^{\infty} [\ell / 2]^{-2(1 + \gamma)/(2 + \gamma)} \xrightarrow[m \to \infty]{} 0
\end{align*}
as \(A_m, B_m \to 0\) by hypothesis.
Hence the second term is also a Cesàro mean converging to zero.
Since the residuals are zero-mean, this shows that they converge to zero in probability, hence vanishing in the distributional limit, and the statement follows.
\end{proof}

\section{Parameter-dependent proposals}\label{supp:parameterdependent}

Using a pathwise estimator (that is, a reparameterization trick, see \barecite{mohamed_monte_2020}) for parameter-dependent proposals accounts for the additional sensitivity, at the cost of introducing new assumptions on the structure of the state space, the existence of gradients of \(f\) and \(\alpha_\theta\), as well as integrability conditions on the pathwise estimator.
This is the approach taken in \texttt{StochasticAD.jl}, as it is a natural extension of automatic differentiation to continuous random variables.

First, an additional weight term in \eqref{eq:weight} is required to account for the sensitivity of the states to \(\theta\), yielding the total derivative chain rule
    \begin{equation*}
        \Tilde{W}_n = W_n + \left(\frac{\partial \alpha_\theta}{\partial x^\circ}(X^\circ_n \mid X_n)\frac{d X_n^\circ}{d \theta} + \frac{\partial \alpha_\theta}{\partial x}(X^\circ_n \mid X_n)\frac{d X_n}{d \theta}\right) \left(1 - 2 \cdot \bm{1}_{U_n \le \alpha_\theta(X^\circ_n \vert X_n)}\right)
    \end{equation*}
where \(\frac{d X_n}{d \theta}, \frac{d X_n^\circ}{d \theta}\) denote the pathwise gradient estimators.
Under additional assumptions on the acceptance probability derivatives in \(x,x^\circ\) similar to \cref{ass:target} as well as moment conditions on the pathwise gradient estimators, one may obtain a new version of \cref{thm:weightmoments}.
Then, \eqref{eq:unbiasedness.IPA} now converges to \(\E[\langle\nabla f(X_k), \frac{d X_k}{d \theta}\rangle]\) by a dominated convergence argument \parencite[Section~1.2]{fu_conditional_1997} and one obtains the extended estimator.
Unfortunately, in general, the pathwise estimator of a single sample could depend on the whole past chain, and the existence of moments of \(\Tilde{W}_n\) is nontrivial, so any consistency argument will likely be ad-hoc.

\subsection{Proposal kernel tuning}\label{supp:proposaltuning}
We apply the above framework to the problem in \cref{sec:proposaltuning}.
In this special case, we will see that significant simplifications are possible.
The Gaussian random walk proposal \(X^\circ \vert X = x \sim \mathsf{N}(x, \theta^2)\) has a direct reparameterization using a latent \(Z \sim \mathsf{N}(0,1)\) as \(X^\circ(X, Z) = X + \theta Z\).
The pathwise derivative estimator for the proposal thus becomes the recursion
\(\frac{d X^\circ}{d \theta} = \frac{d X}{d \theta} + Z\).

First, we consider the contribution from \eqref{eq:unbiasedness.IPA}.
The recursion suggests the pathwise estimator in stationarity should be \(\frac{d X}{d \theta} = \frac{X}{\theta}\).
Indeed, if we take \(\frac{d X_0}{d \theta} = \frac{X_0}{\theta}\) then we have a direct rescaling of the original MH chain and
\begin{equation*}
    \frac{dX_k}{d\theta}(X_k) = \frac{1}{\theta} X_k,\quad \frac{d X^\circ_k}{d \theta}(X^\circ_k) = \frac{1}{\theta} X_k^\circ
\end{equation*}
where we can write the derivative estimates as functions of the current state,
allowing us to trivially transfer the ergodicity properties of the augmented chain to the sequence of pathwise estimators.
Hence the pathwise derivative appearing in \(\widehat{\frac{\partial \gamma_1}{\partial \theta}}\) is
\begin{equation*}
    \frac{d}{d\theta}(X_k X_{k+1}) = \frac{d X_k}{d \theta}X_{k+1} + X_k\frac{d X_{k+1}}{d \theta} = \frac{2}{\theta} X_k X_{k+1}
\end{equation*}
with consistency for this term following from the usual ergodic theorems, as long as the target has second moments.
In actual simulation, we may then start with \(\frac{d X_0}{d \theta} = 0\) as would be implied by a constant starting state and still have a guarantee that this term converges.

Second, we consider the contribution from \eqref{eq:unbiasedness.SPA}.
Recall that \(g\) denotes the unnormalized target density and observe that we, using the symmetry of the proposal, obtain the acceptance probability derivative following the previous discussion
\begin{align*}
    \frac{d}{d\theta} \alpha_\theta(X_n^\circ \vert X_n) &= \frac{\partial \alpha_\theta}{\partial x^\circ}(X_n^\circ \vert X_n) \frac{d X_n^\circ}{d\theta} + \frac{\partial \alpha_\theta}{\partial x}(X_n^\circ \vert X_n) \frac{d X_n}{d\theta} \\
    &= \bm{1}_{\alpha_\theta(X_n^\circ \vert X_n) < 1} \frac{\alpha_\theta(X_n^\circ \vert X_n)}{\theta} \left(\frac{\partial}{\partial x}\left[\log g(X_n^\circ)\right]X_n^\circ - \frac{\partial}{\partial x}\left[\log g(X_n)\right]X_n\right)
\end{align*}
for which the analogue of \cref{thm:weightmoments} follows if \(\E_{X \sim \pi}[\lvert X \frac{\partial}{\partial x} \log g(X) \rvert^2] < \infty\).
For the second moment to exist of \(f(X_k,X_{k+1}) = X_k X_{k+1}\) (to conclude \cref{thm:h}) it is clear that it is sufficient that the target distribution possesses fourth moments.
Then, following the discussion on performance functionals dependent on multiple states, we will have consistency arguing similarly to the proofs of \cref{thm:unbiasedness,thm:consistency}.
Summing these two terms together to form the estimator, unbiasedness follows from the previous discussion, and consistency follows from the aforementioned arguments.

The multidimensional case is handled analogously; the entries of the cross-covariance matrix are essentially component-wise products, so the derivative of the determinant can be related to the derivative of each entry by Jacobi's formula.
Similarly, a full parameter-dependent covariance matrix will also lead to a (more complicated) rescaling of the original chain as the pathwise estimator.

\FloatBarrier%
\section{MCMC Diagnostics for Section~\ref{sec:proposaltuning}}\label{app:proposaltuning.diag}
The diagnostics for the final proposals in \cref{sec:proposaltuning} (and comparisons to the naïve manual tuning) were obtained by running 4 chains of 250\,000 steps starting from the same state as the optimizer.
The definitions of effective sample size (ESS) and \(\hat R\) follow those of \textcite{vehtari_rank-normalization_2021}.

\begin{table}[h]
    \centering
    \caption{Gaussian with unit variance and correlation \(0.5\), suggested proposal. Final acceptance rate 0.354.}
    \label{tab:proposaltuning.gauss}
    \begin{tabular}{r S[table-format=1.4] S[table-format=1.4] S[table-format=1.4] S[table-format=6.4] S[table-format=6.4] S[table-format=1.4]}
    \toprule
     Param. & {Mean} & {Std.} & {MC s.e.} & {ESS (bulk)} & {ESS (tail)} & {\(\hat R\)} \\\midrule
     \(x_1\) & -0.0004 & 0.9995 & 0.0028 & 128529.3669 & 171540.5235 & 1.0000 \\
     \(x_2\) & 0.0033 & 0.9976 & 0.0027 & 136899.0162 & 177234.6313 & 1.0000 \\\bottomrule
    \end{tabular}
\end{table}
\begin{table}[h]
    \centering
    \caption{Dual Moon, suggested proposal. Final acceptance rate 0.166.}
    \label{tab:proposaltuning.dualmoon}
    \begin{tabular}{r S[table-format=1.4] S[table-format=1.4] S[table-format=1.4] S[table-format=6.4] S[table-format=6.4] S[table-format=1.4]}
    \toprule
     Param. & {Mean} & {Std.} & {MC s.e.} & {ESS (bulk)} & {ESS (tail)} & {\(\hat R\)} \\\midrule
     \(x_1\) & -0.0036 & 1.6003 & 0.0085 & 39979.9761 & 82117.3025 & 1.0001 \\
     \(x_2\) & -0.0016 & 1.6020 & 0.0083 & 42611.9025 & 86820.3302 & 1.0001 \\\bottomrule
    \end{tabular}
\end{table}
\begin{table}[h]
    \centering
    \caption{Dual Moon, naïve proposal. Final acceptance rate 0.234.}
    \label{tab:proposaltuning.dualmoon.naive}
    \begin{tabular}{r S[table-format=1.4] S[table-format=1.4] S[table-format=1.4] S[table-format=6.4] S[table-format=6.4] S[table-format=1.4]}
    \toprule
     Param. & {Mean} & {Std.} & {MC s.e.} & {ESS (bulk)} & {ESS (tail)} & {\(\hat R\)} \\\midrule
     \(x_1\) & -0.0025 & 1.6047 & 0.0110 & 25431.7478 & 90642.5985 & 1.0001 \\
     \(x_2\) & -0.0043 & 1.5966 & 0.0112 & 23931.9972 & 88170.5031 & 1.0002 \\\bottomrule
    \end{tabular}
\end{table}
\begin{table}[h]
    \centering
    \caption{Rosenbrock-type, suggested proposal. Final acceptance rate 0.156.}
    \label{tab:proposaltuning.rosenbrock}
    \begin{tabular}{r S[table-format=1.4] S[table-format=1.4] S[table-format=1.4] S[table-format=6.4] S[table-format=6.4] S[table-format=1.4]}
    \toprule
     Param. & {Mean} & {Std.} & {MC s.e.} & {ESS (bulk)} & {ESS (tail)} & {\(\hat R\)} \\\midrule
     \(x_1\) & -0.0049 & 0.4492 & 0.0029 & 24031.4766 & 18120.7512 & 1.0002 \\
     \(x_2\) &  0.2019 & 0.3003 & 0.0025 & 27835.1758 & 15459.7793 & 1.0002 \\\bottomrule
    \end{tabular}
\end{table}
\begin{table}[h]
    \centering
    \caption{Rosenbrock-type, naïve proposal. Final acceptance rate 0.233.}
    \label{tab:proposaltuning.rosenbrock.naive}
    \begin{tabular}{r S[table-format=1.4] S[table-format=1.4] S[table-format=1.4] S[table-format=6.4] S[table-format=6.4] S[table-format=1.4]}
    \toprule
     Param. & {Mean} & {Std.} & {MC s.e.} & {ESS (bulk)} & {ESS (tail)} & {\(\hat R\)} \\\midrule
     \(x_1\) & -0.0011 & 0.4497 & 0.0034 & 17109.1143 & 18581.6855 & 1.0002 \\
     \(x_2\) &  0.2023 & 0.3013 & 0.0025 & 29602.4486 & 16926.9068 & 1.0002 \\\bottomrule
    \end{tabular}
\end{table}

\end{appendix}

\begin{acks}[Acknowledgments]
We warmly thank Paolo Ceriani, Adrien Corenflos, Peter Glynn, Bernd Heidergott, Alexander Lew, Frank Schäfer, and Aila Särkkä for fruitful discussions.\end{acks}

\bibliographystyle{imsart-number}
\bibliography{DMH}

\end{document}